\documentclass[11pt,a4paper]{article}
\usepackage[dvips]{graphicx}
\usepackage[utf8]{inputenc}
\usepackage{enumerate}
\usepackage[spanish,english]{babel}
\usepackage{amsthm,amssymb,amsfonts,amsmath}
\usepackage[ruled,linesnumbered]{algorithm2e}
\usepackage{lineno}
\usepackage{xcolor}
\usepackage{cite}
\usepackage{pifont}
\usepackage{enumerate}
\usepackage{hyperref}
\usepackage{listings}
\usepackage{xcolor}
\usepackage{authblk}

\newtheorem{theorem}{Theorem}[section]

\newtheorem{proposition}[theorem]{Proposition}

\newtheorem{lemma}[theorem]{Lemma}
\newtheorem{corollary}[theorem]{Corollary}
\newtheorem{observation}[theorem]{Observation}

  \theoremstyle{definition}
\newtheorem{remark}[theorem]{Remark}

\graphicspath{{figures/}}

\newcommand{\Nl}{\ensuremath{\mathcal{G}}\xspace}
\newcommand{\T}{\ensuremath{\mathcal{T}}\xspace}

\newcommand{\diam}{\ensuremath{{\rm diam}}\xspace}

\title{The Borsuk number of a graph}
\author[1]{José Cáceres}
\affil[1]{Universidad de Almería, Spain} 
\author[2]{Delia Garijo}
\affil[2]{Universidad de Sevilla, Spain} 
\author[2]{Alberto Márquez}
\author[3]{Rodrigo I. Silveira}
\affil[3]{Universitat Politècnica de Catalunya, Barcelona, Spain}

\date{}

\begin{document}

\maketitle

\begin{abstract}
The \emph{Borsuk problem} asks for the smallest number of subsets with strictly smaller diameters into which any bounded set in the $d$-dimensional space can be decomposed. It is a classical problem in combinatorial geometry that has been subject of much attention over the years, and research on variants of the problem continues nowadays in a plethora of directions. In this work, we propose a formulation of the problem in the context of graphs. Depending on how the graph is partitioned, we consider two different settings dealing either with the usual notion of diameter in abstract graphs, or with the diameter in the context of continuous graphs,  where all points along the edges, instead of only the vertices, must be
taken into account when computing distances. We present complexity results, exact computations and upper bounds on the parameters associated to the problem.
\end{abstract}

\section{Introduction}

In 1933, Borsuk posed the question of whether every bounded set $X$ in $\mathbb{R}^d$ could be partitioned into $d+1$ closed (sub)sets each with diameter  smaller than that of $X$~\cite{borsuk1933three}.\footnote{In this context, the \emph{diameter} is defined as the supremum of the distances between any two points in the set, under the Euclidean metric.} This question can be reformulated in terms of the so-called  \emph{Borsuk number}.
For a set $X \subset \mathbb{R}^d$, the Borsuk number $b(X)$ is the minimal number of subsets with strictly smaller diameters into which $X$ can be decomposed. Figure~\ref{fig:square} illustrates an example.

Borsuk's question can be thus stated as whether $b(X) \leq d+1$, for any bounded set $X \subset \mathbb{R}^d$.
The answer was shown to be positive for $d=2,3$~\cite{eggleston55covering,perkal47sur}, 
and for general $d$ for centrally symmetric convex bodies~\cite{Rissling71Das}
and smooth convex bodies~\cite{hadwiger1946mitteilung}. 
The general answer turned out to be negative, as shown in 1993 by Kahn and Kalai~\cite{kahn1993counterexample}. 
Since then, researchers have been trying to figure out the smallest dimension for which the partition does not exist, being $d=64$ the currently best~\cite{JenrichB14}. Many variants of the  Borsuk problem have also been studied, see~\cite{zong21survey} for a recent survey.

\begin{figure}
\centering
    \includegraphics{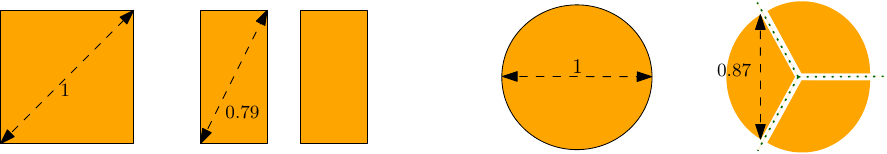}
     \caption{A square with diameter $1$ can be divided into two pieces with smaller diameter ($\approx  0.79$), and therefore, its Borsuk number is $2$. However, the circle has to be divided into at least three pieces to decrease the diameter, thus it has Borsuk number $3$. }
	\label{fig:square}
\end{figure}

In this paper, we present a formulation of Borsuk's problem in the context of graphs. Conceptually, we define the \emph{Borsuk number $b(G)$ of a graph $G$} in a similar fashion to the original problem:  the smallest number of subgraphs with smaller diameter into which $G$ can be partitioned. However, we need to define carefully how a graph can be partitioned.
We propose two natural ways to do this, which lead to two variants of the problem, the \emph{continuous} and the \emph{discrete} variants.  

The discrete variant deals with distances in abstract graphs. The partition of the graph is done by deleting edges, and we consider the usual notion of \emph{diameter} in graphs, that is,  the maximum distance between any two vertices. This parameter has been studied extensively over the years in many directions; in particular, its computation is a fundamental problem in algorithmic graph theory; see, for instance~\cite{RodittyW13, BergeDH24, BergeDH25, BHM20, Cabello19, Ducoffe2022, RoddichowskiKMS21, Le2024}. 

The continuous version of the problem partitions a \emph{continuous geometric graph}
by cuts with straight lines.
A \emph{geometric graph} is an undirected graph  whose vertices are points in $\mathbb{R}^2$, and whose edges are straight-line segments connecting pairs of points.  When there are
no crossings between the segments (except at their endpoints), the geometric graph is said to be \emph{plane}. 
The vertices and the edges of a plane geometric graph $G$ determine an infinite set of points in the plane, where each point on an edge is considered to be part of the graph. This set can be seen as a 1-dimensional simplicial complex, which we call the \emph{continuous geometric graph defined by $G$}, or simply, \emph{continuous geometric graph}, and denote it by $\Nl$. We use the term ``continuous" to stress that all points along the edges of $G$ are part of $\Nl$, and therefore, the diameter of $\Nl$ is  the maximum distance between any two points along the edges, rather than between any two vertices. As in the discrete variant, the efficient computation of this parameter is notably difficult; see~\cite{CGKKPS25} for very recent results on the topic.

The diameter in graphs plays a central role in problems related with network design and network optimization. How to modify a graph in order to optimize a certain parameter is a problem that is currently being actively studied, both in the discrete version and in the continuous one. Probably the most studied problem in this context is the \emph{diameter-optimally augmentation problem}, whose goal is, informally, to add a given number of new connections to a graph such that the diameter of the resulting graph is minimized. This problem is especially challenging not only in the continuous setting (where the new connections may join any two points along the edges of the graph) but also in the discrete case (where the connections are new edges connecting two vertices.) See, for example,~\cite{BiloGLS23, FGGM14, GGKSS2015} (and the references therein) for results in the discrete setting, and~\cite{bae2019shortcuts,CACERES2018192,DBLP:conf/swat/CarufelMS16, CGSS-20, GarijoMRS19} for studies in the continuous case\footnote{We note that continuous graphs, which are also called  \textit{metric graphs} in other areas closer to analysis~\cite{Lumer_1980,Friedlander_2005}, are defined, in general, by graphs embedded in metric spaces~\cite{CGKKPS25, MeanDist23}. In this paper, we restrict our attention to those defined by plane geometric graphs, which are also being intensively studied, mainly in the context of graph augmentation~\cite{CACERES2018192,DBLP:conf/swat/CarufelMS16, CGSS-20, GarijoMRS19}.}.

The Borsuk's problem for graphs that we introduce in this paper also deals with the computation of the diameter while modifying a graph. Our study is undertaken in both, the discrete setting and the continuous one. A significant contribution of our work is that the key ideas to prove our main results are based on an accurate analysis of how shortest paths and distances can change when modifying a graph.

\paragraph{Our results.}
After introducing the required definitions and notation in Section~\ref{sec:preliminaries}, we define the two variants of Borsuk number of a graph: the continuous variant (Section~\ref{sec:cbn}) and the discrete variant (Section~\ref{sec: definition discrete}). In analogy with Borsuk’s original problem, in both subsections, we discuss whether the Borsuk number can be upper bounded by a constant. Although this question remains open in the continuous setting, we make some non-trivial progress, presenting an upper bound that is linear on the number of vertices of the graph. In the discrete case, we construct examples of graphs with linear  Borsuk number (in the number of vertices.)  

Section~\ref{sec:complexity} addresses complexity questions. We prove the NP-completeness of the problem of deciding whether the Borsuk number of a graph is below a given threshold, in both the discrete and the continuous setting. Our study relates Borsuk's problem in graphs to other well-known problems in graph theory and discrete geometry: the \emph{minimum clique cover} problem, and a variant of the \emph{point line cover} problem.

In Section~\ref{sec:monotone}, we further explore the question of how large the Borsuk number of a  graph might be. Concretely, we study a
large class of continuous geometric graphs that satisfy a natural property of monotonicity with respect to the intersecting lines. Our main result establishes a non-trivial upper bound for their  Borsuk number which, informally, depends on the structure of shortest paths of a continuous geomeric graph. 

Section~\ref{sec:trees} contains our study on trees. We show that, while the discrete variant of the Borsuk number can be linear in the number of vertices of the graph, in the continuous setting, this parameter is upper bounded by $3$, and can be computed in quadratic time in the number of vertices of the tree. 

We conclude in Section~\ref{sec:conclusions} summarizing some intriguing open questions.

\section{Preliminaries}\label{sec:preliminaries}

Let $G=(V,E)$ be a plane geometric graph, and let $\Nl$ be its associated continuous geometric graph. Each edge $e=uv$ of $G$ has a length, $|e|$, equal to the Euclidean distance between its endpoints $u,v$. Since all geometric graphs considered in this work are plane, we omit this term throughout the paper, as well as the term ``continuous" when no confusion arises.

Since $\Nl$ can be seen as a 1-dimensional simplicial complex, we shall indistinctly say vertices and edges of $G$ or $\Nl$. In addition, we shall use $p\in \Nl$ to indicate that $p$ is a point of $\Nl$, and  $p\in e$ for edges $e$.

A point $p\in \Nl$ can be specified by a triple $(uv,u,\lambda)\in E\times V\times [0,|uv|]$, which represents the point of the segment $uv$ at distance $\lambda$ (along $uv$) from the endpoint $u$. The triples $(uv,u,\lambda)$ and $(uv,v,|uv|-\lambda)$ define the same point of $\Nl$.

\subsection{Diameter: discrete and continuous variants}

A \emph{path} $P$ in a graph is a sequence of vertices and edges, where each edge connects two consecutive vertices in the sequence. The \emph{length} of the path, $|P|$, is the sum of the lengths of its edges. Notice that, while in a geometric graph, the length--edges are determined by the Euclidean distance between the segments endpoints, in an abstract graph, all length--edges are equal to $1$.
The \emph{distance}  between two vertices $u,v$ in a graph $G=(V, E)$, denoted by $d_G(u,v)$, is the length of a shortest path connecting them; we use $d_T(u,v)$ when the graph is a tree $T$. 
The \emph{diameter} of $G$ is defined as
$\diam(G) :=\max_{u,v\in V} d_G(u,v)$.

The concepts of path and distance naturally extend from a geometric graph $G$ to a continuous geometric graph $\Nl$. In the continuous setting, paths are defined not only by edges, but also by at most two fragments of edges (which are also segments with their usual metric
and measure). For points $p,q \in\Nl$, we use $pq$-path to refer to a path with endpoints $p,q$, and $d_{\Nl}(p, q)$ to refer to their distance ($d_{\T}(p, q)$ for continuous geometric trees $\T$). 
The distance between the two endpoints of an edge $e$ is $|e|$. Thus, if two points $p = (uv, u, \lambda)$ and $q = (uv, u, \lambda')$ lie on the same edge $uv$, then $d_{\Nl}(p, q) = |\lambda - \lambda'|$.

The \emph{diameter}\footnote{The diameter of a continuous graph is also called \emph{generalized diameter}~\cite{CG82}.} of $\Nl$ is defined as ${\rm diam}(\Nl):=\max_{p,q\in \Nl} d_\Nl(p,q)$. Two points whose distance attains this value are called \emph{diametral points} or \emph{diametral pair} (of points), and the shortest paths connecting diametral points are called \emph{diametral paths}; the analogous terms are used for vertices and paths in $G$. 

Notice that the notation ${\rm diam}(G)$ refers to the \emph{discrete} version of this parameter (meaning to consider distances only between the vertices of $G$), while the notation ${\rm diam}(\Nl)$  indicates the \emph{continuous} variant (distances are taken between every two points along the edges of $G$).

Figure~\ref{fig:diameter_shortcut_worse} illustrates how the diameter of a graph can change depending on whether we consider the discrete or the continuous  variant. 
In the discrete case, when adding a new edge, the diameter  
can remain the same or decrease  
(the latter occurs if and only if the edge shortens the distance between all diametral pairs), but it cannot  increase. 
In contrast, in the continuous case, the diameter can increase after an edge insertion.
This is because, in continuous graphs, all points on the new edge participate in the diameter computation, thus the  inserted edge can contain a point whose farthest point is at a larger distance than the original diameter.
An example of this---rather counterintuitive---situation is shown in Figure~\ref{fig:diameter_shortcut_worse}.

\begin{figure}
\centering
    \includegraphics{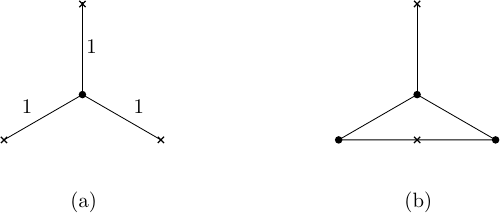}
     \caption{ (a) A graph with diameter $2$: the value is the same in both settings, the discrete and the continuous; the distance between any two vertices with a cross mark gives the diameter.
     (b) Adding one edge to the graph in (a): the diameter in the continuous setting is now given by a vertex and the midpoint of the new segment; its value increases ($\approx  2.86$).}
	\label{fig:diameter_shortcut_worse}
\end{figure}

Another difference between both settings is that, in contrast to abstract graphs, in a continuous graph, there can be an infinite number of diametral points. For example, the diameter of any (continuous) cycle $\mathcal{C}$ is $|\mathcal{C}|/2$, where $|\mathcal{C}|$ denotes the perimeter or length of the cycle; moreover, every point on the cycle belongs to a diametral pair, see Figure~\ref{fig:infinite_diameters}.

\begin{figure}
\centering
	\includegraphics{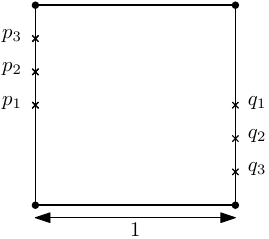}
     \caption{A square with side length $1$. The diameter of this cycle is $2$, and  every point belongs to a diametral pair: $(p_1,q_1), (p_2,q_2), (p_3,q_3), \ldots$}
	\label{fig:infinite_diameters}
\end{figure}

\section{Definitions of Borsuk number}\label{sec:definitions}

\subsection{Continuous variant}\label{sec:cbn}

We partition a continuous geometric graph $\Nl$ by a sequence of cuts with straight lines.

A line $\ell$ naturally divides $\Nl$ into subgraphs. 
However, by cutting one connected graph along $\ell$ we may obtain multiple subgraphs  (see Figure~\ref{fig:star}), thus distorting the core idea of Borsuk's original problem. To remain in the same spirit, our aim would be that each cut in a connected graph produces two connected subgraphs. 
To guarantee this, after cutting we add to the subgraphs obtained in each halfplane (determined by $\ell$) the longest segment in $\ell$ that has its endpoints in $\Nl \cup \ell$; this maximal segment is denoted by
 $s_{\ell}$. 
 In this way, the partition results in exactly two connected geometric subgraphs of $\Nl \cup \ell$, which are defined as:
$$\Nl_{\ell}^+:=(\ell^+\cap \Nl)\cup s_{\ell} \quad {\rm and} \quad  \Nl_{\ell}^-:=(\ell^-\cap \Nl)\cup s_{\ell},$$
where $\ell^+$ and $\ell^-$ are, respectively,  the open half-planes above and below  $\ell$ (right and left when the line is vertical). We also insert in both subgraphs, $\Nl_{\ell}^+$ and $\Nl_{\ell}^-$, a new vertex at each intersection point of $\ell$ with $\Nl$. Figure~\ref{fig:star} illustrates the partition method by a line $\ell$.

In addition to preserving the original idea of Borsuk's  problem, our partitioning method has natural applications in network design: given a road network, the goal would be to construct a new road (crossing it) to decrease the maximum distance in each of the resulting parts (compared to the original network); the new road would be our segment $s_{\ell}$.

An ordered set of lines $(\ell_1, \ldots, \ell_k)$ partitions a graph $\Nl$ sequentially:  each line $\ell_i$ is inserted only into one of the existing subgraphs, producing two new subgraphs of the types $\Nl_{\ell}^+$ and $\Nl_{\ell}^-$. At the end of the process, one obtains $k+1$ geometric subgraphs of $\Nl \cup \ell_1 \cup \ldots \cup \ell_k$. To simplify  notation, we denote them by $\Nl_1,\ldots ,\Nl_{k+1}$, omitting the subgraph that has been partitioned at that stage of the process, and the line used to partition.

\begin{figure}
\centering
	\includegraphics[width=13cm]{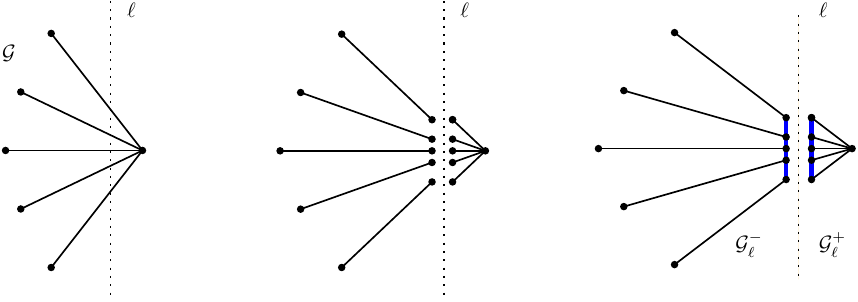}
     \caption{Partitioning a graph $\Nl$ by a line $\ell$ to obtain the subgraphs $\Nl_{\ell}^+$ and $\Nl_{\ell}^-$ (segment $s_{\ell}$ in thicker blue.)}
	\label{fig:star}
\end{figure}

We define the \emph{Borsuk number of $\Nl$}, and denote it by $b(\Nl)$, as the minimum cardinality of a partition of $\Nl$ by an ordered set of lines $(\ell_1, \ldots, \ell_k)$ into  subgraphs $\Nl_1,\ldots ,\Nl_{k+1}$ (of $\Nl \cup \ell_1 \cup \ldots \cup \ell_k)$ such that: $$\max\{\diam(\Nl_1),\ldots ,\diam(\Nl_{k+1})\}< \diam(\Nl).$$
Figure~\ref{fig:example}a and Figure~\ref{fig:example}b show examples of this definition: after partitioning the square  with a vertical line $\ell$ through its center point, we obtain two subgraphs (of $\Nl \cup \ell$) 
that have a diameter smaller than that of $\Nl$, and so $b(\Nl)=2$ (best possible).
The example in Figure~\ref{fig:example}b shows a 4-star graph, which requires at least two lines to decrease the diameter of the original graph, giving Borsuk number three. 

Note that the definition of  Borsuk number (in this continuous setting) does not require the graph $\Nl$ to be connected. For a disconnected $\Nl$, which has infinite diameter, any partition by lines must in fact connect the components in order to reduce the diameter of the original graph to a finite value. In Section~\ref{sec:complexity}, we use a construction with a disconnected $\Nl$ to address a complexity question, but in the remainder of the paper we assume that $\Nl$ is connected, as this is a natural property of the problem.

We shall use the term \emph{correct partition} to refer to a partition of a graph in which all resulting subgraphs indeed have a smaller diameter than that of the original graph.

\begin{figure}
\centering
	\includegraphics{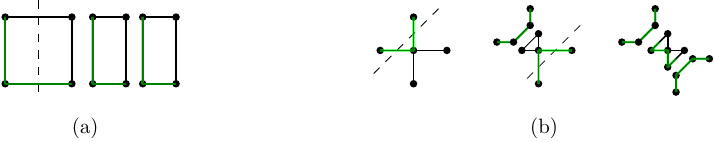}
     \caption{(a) A square with side length 1 and diameter 2 (given by green paths), and a partition with a  line (dashed). (b) A 4-star partitioned into three subgraphs.} 
	\label{fig:example}
\end{figure}

As Borsuk's conjecture holds for bounded sets in the plane, that is, $b(X) \leq 3$ for any bounded set $X \subset \mathbb{R}^2$, one of the main open questions in this continuous setting is whether $b(\Nl)$ can be upper bounded by a constant. The following proposition constitutes a first answer to this question, giving an upper bound that is linear on the number of vertices of $\Nl$.

\begin{proposition}\label{prop:firstbound}
Given a continuous geometric graph $\Nl$ with $n\geq 3$ vertices, $b(\Nl)\leq n-1$. 
\end{proposition}

\begin{proof}
Consider a direction that is not parallel to any of the edges of $\Nl$; assume, for simplicity, that this is the vertical direction. 
We divide $\Nl$ by $n$ vertical lines, one through each vertex, to create $n-1$ graphs, denoted by $\Nl_1, \ldots, \Nl_{n-1}$. Thus, there are $n-1$ vertical strips, each containing two vertices of the original graph, and all edges go from the initial vertical line defining the strip to the other vertical line of the strip.

We are going to prove that in each $\Nl_i$, there are two points of the original graph realizing the diameter of $\Nl_i$, and that their distance is smaller than the diameter of $\Nl$.

Consider a graph $\Nl_i$, and a diametral pair of $\Nl_i$ formed by $p$ and $q$. 

If both $p$ and $q$ are part of $\Nl$, either they belong to the same edge of $\Nl$ or to two different edges. 
If they lie on the same edge, since $n \geq3$, their distance must be smaller than ${\rm diam}(\Nl)$.
Otherwise, their distance in $\Nl_i$ is smaller than their distance in $\Nl$, because the vertical edges of $\Nl_i$ do not exist in $\Nl$.
Therefore, we have that $d_{\Nl_i}(p,q)$ is strictly smaller than ${\rm diam}(\Nl)$.

Otherwise, at least one of $p$ or $q$ lies on one of the two vertical lines defining $\Nl_i$.

Since there are no vertices in the interior of $\Nl_i$, the shortest path between $p$ and $q$ is not unique: there must exist two shortest paths between $p$ and $q$ such that the union of both paths is a trapezoid (with the parallel edges being segments of the vertical lines) or a triangle (with one vertical side). 
In the following, refer to Figure~\ref{fig:bandas}.

In the case of the trapezoid, consider two diagonally opposite vertices of the trapezoid, $p'$ and $q'$, and the longest path along that trapezoid joining them.
That path contains a vertical segment, that without loss of generality we assume that starts  at $p'$. 
Then, if $p''$ is the diametral point to $p'$ in the trapezoid, it is straightforward to see that $p', p'' \in \Nl$, and that $d_{\Nl_i}(p,q)=d_{\Nl_i}(p',p'')$. But 
$d_{\Nl}(p',p'')>d_{\Nl_i}(p',p'')$.

In the case of the triangle, let $u,p,v$ be its vertices.
Assume without loss of generality that $u$ and $p$ lie on the same vertical line bounding $\Nl_i$, thus $v$ lies on the other line---and must be a vertex of $\Nl$.
The diameter of $\Nl_i$ is given by any point on this triangle and its corresponding diametral point, also on the triangle.
Consider $u$ and its diametral point $u'$.
Now, consider a point $u''$   very close to $u'$, on the same edge as $u'$, such that the shortest path in $\Nl_i$ between $u$ and $u''$ goes through $p$ (and the vertical segment of the triangle).
Both $u$ and $u''$ exist in $\Nl$, but their distance is larger there, since the vertical segment does not exist in $\Nl$.
Then, ${\rm diam}(\Nl_i )=d_{\Nl_i}(u,u')<d_{\Nl}(u,u'') \leq {\rm diam}(\Nl)$.
\end{proof}

\begin{figure}
\centering
	\includegraphics{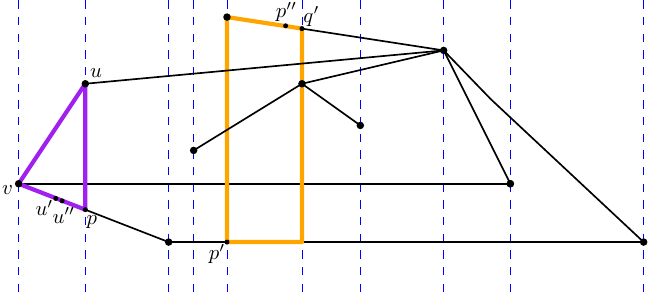}
     \caption{Illustration of Lemma~\ref{prop:firstbound}. Splitting a graph $\Nl$ by vertical lines to create the graphs $\Nl_i$. In orange we illustrate the trapezoid case of the proof, while in purple we show the triangle case.}
	\label{fig:bandas}
\end{figure}

\subsection{Discrete variant}\label{sec: definition discrete}

We now consider partitions of a  connected abstract graph $G$ by deleting edges; here all edges have the same length, equal to 1. (We do not deal with the case of deleting vertices as our interest relies on the distances between vertices of the original graph.)

Edge deletion is a fundamental operation in graph theory. As in the continuous variant, this partitioning method preserves the core idea of Borsuk's original problem: by deleting one edge, we generate at most two subgraphs. In addition, there are multiple applications related with edge deletion; for example, this operation can be seen as removing some connection between the individuals in a social network, where the distance is a key parameter to measure the behavior of the network.

We define the \emph{Borsuk number of a (connected abstract) graph $G$}, and denote it by $b(G)$, as the minimum cardinality of a partition of $G$ (by deleting edges) into subgraphs $G_1,\ldots ,G_k$ (of $G$) such that: $$\max\{\diam(G_1),\ldots ,\diam(G_k)\}< \diam(G).$$
Recall that, in the discrete variant, the diameter of a graph is the maximum distance between any two vertices. 

The following observation exhibits some simple examples. 

\begin{observation}\label{obs:examples}
\begin{enumerate}
\item[(i)] If $G$ is a path or a cycle of even length, $b(G)=2$.
\item[(ii)] If $G$ is a cycle of odd length, $b(G)=3$.
    \item[(iii)] If $G$ is a star graph on $k+1$ vertices, $b(G)=k$.
\end{enumerate}
\end{observation}

\begin{proof}
For paths and cycles of even length it suffices to remove, respectively, one and two edges to obtain two subgraphs with a diameter smaller than the original graph, and so the Borsuk number is $2$. If we delete only two edges in a cycle of odd length, one of the resulting subgraphs will have the same diameter as $G$; this forces to delete three edges giving rise to three subgraphs, leading to Borsuk number $3$.
To decrease the diameter of a star, one needs to remove all but one edge, thus for a star with $k+1$ vertices, this generates  a partition into $k$ subgraphs, leading to Borsuk number $k$. 
\end{proof}

In Section \ref{sec:trees}, we study the Borsuk number of an arbitrary tree, in both the discrete and the continuous settings. We show that while $b(\mathcal{T})$ is upper bounded by a constant for any continuous geometric tree $\mathcal{T}$, in the discrete case, the Borsuk number can be linear in the number of vertices of the graph, as it happens for the star. This linearity also occurs in other families of graphs, such as unicycle graphs (which comes from the scenario in trees) and maximal outerplanar graphs that are not trees. Figure~\ref{fig:outerplanar} shows examples of the latter case: the graph on the left has diameter two, so the connected components obtained after removing the edges must have either diameter one or zero; this implies the deletion of at least $n-2+\lfloor \frac{n-1}{2}\rfloor$ edges for a fan with $n+1$ vertices, resulting in Borsuk number $\lceil \frac{n}{2}\rceil$. This argument can be adapted for maximal outerplanar graphs with diameter larger than two, such as the example illustrated in Figure~\ref{fig:outerplanar}(right). 

\begin{figure}
\centering
	\includegraphics{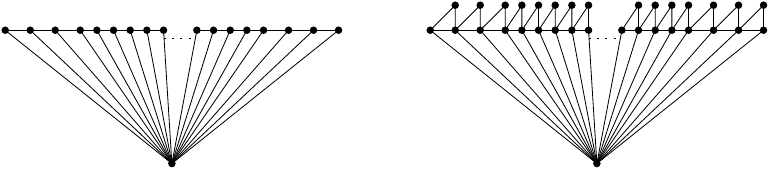}
     \caption{Examples of maximal outerplanar graphs $G$ with linear  Borsuk number $b(G)$.}
	\label{fig:outerplanar}
\end{figure}

\section{Complexity}\label{sec:complexity}

In this section, we prove that deciding whether the Borsuk number of a graph is below a given threshold is an NP-complete problem, in both the discrete and the continuous variants.

In the discrete setting, this problem is related to the  \emph{minimum clique cover problem}.
A \emph{clique cover} of a graph $G$ is a partition of its vertex set into cliques (i.e., subsets of vertices that induce a complete graph.) The \emph{clique cover number} of $G$, denoted by $\theta(G)$, is the minimum size of a clique cover, that is, the minimum number of cliques that cover all  vertices. The minimum clique cover problem seeks for a minimum clique cover.

\begin{lemma} \label{lem:diam2}
If $G$ is a non-complete connected graph, then $b(G)\leq\theta(G)$. Furthermore, if $\diam(G)=2$ then $b(G)=\theta(G)$.
\end{lemma}
  
\begin{proof}
The vertex set of $G$ can be decomposed into $\theta(G)$ disjoint subsets, each being a clique. By removing from $G$ all the edges connecting the $\theta(G)$ distinct induced subgraphs by the cliques, we obtain $\theta(G)$ subgraphs, each with diameter at most 1. Therefore, $b(G) \leq \theta(G)$.

If $\diam(G)=2$, all the subgraphs of any partition determining the Borsuk number $b(G)$ must have diameter at most $1$. Hence, the vertex set of each subgraph must be a clique. Then, any such partition determines a clique cover of $G$; this implies $\theta(G)\leq b(G)$, which gives $\theta(G) = b(G)$ since $G$ is a non-complete graph.
\end{proof}

A reduction to the decision version of the minimum clique cover problem, which  is a classical NP-complete problem \cite{K72}, allows us to establish the analogous result for the Borsuk number.

\begin{theorem}  \label{th:dbnnp}
Let $G=(V,E)$ be a connected graph, and let $k$ be a positive integer number. The problem of deciding whether  $b(G) <k$ is  NP--complete.   
\end{theorem}
\begin{proof}
 Let $G$ be a graph such that ${\rm diam}(G)\geq 2$ (otherwise, $G$ is a complete graph, and so $b(G)=|V|$, since all edges must be removed to decrease the diameter from $1$ to zero). 
 
 The \emph{cone} $C_G$ of $G$ is the graph obtained from $G$ by adding a new vertex adjacent to all the vertices in $G$. Note that $b(C_G)\leq b(G)+1$: simply consider the partition determining the Borsuk number of $G$, and remove all previously added edges to obtain a partition with one more subgraph (the new vertex).
 
 Since ${\rm diam}(C_G)=2$, by Lemma~\ref{lem:diam2} we have $\theta(C_G)=b(C_G)$. 
 Further, the graph $G$ has a clique cover of size $k$ if and only if $C_G$ has a clique cover of size $k$. Thus, we have $b(G)\leq \theta(G)=\theta(C_G)=b(C_G)\leq b(G)+1$ (the first inequality given by Lemma~\ref{lem:diam2}). The result then follows from the NP-completeness for the minimum clique cover~\cite{K72}, which is known to be NP-complete even for graphs of diameter 2.
\end{proof}

In the continuous setting, the decision problem for the Borsuk number is related to the  \emph{point line cover problem}: for a given set of points in the plane, the goal is to find a set of lines of minimum cardinality such that each point lies on at least one line. This is a well-studied problem in discrete geometry, which was shown to be NP-hard by Megiddo and Tamir in 1982~\cite{MegiddoT82}. Surprisingly, to the best of our knowledge, variants of the point line cover problem where the objects used to cover the points are disjoint have only been studied recently in~\cite{Grelier22, gms-2026}. The authors in~\cite{gms-2026} use either disjoint line segments (also used in~\cite{Grelier22}) or guillotine cuts (which are, essentially, also disjoint) to cover the points. See Figure~\ref{fig:example-problems} for an example of both problems. In particular, they prove the following theorem, which allows us to prove the NP-completeness of the decision problem for the Borsuk number.

\begin{figure}[htbp]
    \centering
    \includegraphics{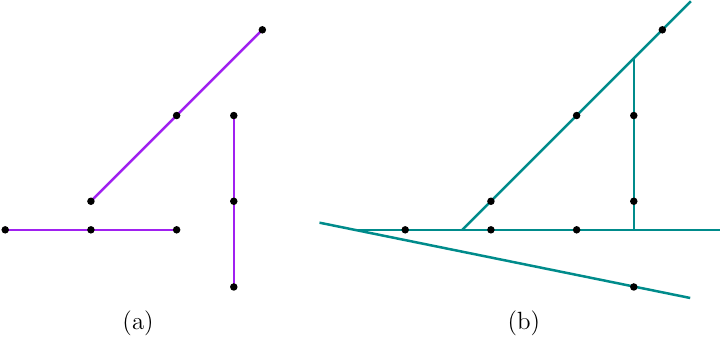}
    \caption{Optimal coverings for $n=9$ points with (a) three disjoint segments, (b) four guillotine cuts (this way of covering inserts at each step either a line or half-line or  segment, with the restriction that each inserted element divides only one existing region into two.)}
    \label{fig:example-problems}
\end{figure}

\begin{theorem}[\cite{gms-2026}] \label{th:guillotine}
Given a set $S$ of points in the plane and a positive integer $k$, the problem of determining if all points in $S$ can be covered by $k$ guillotine cuts is NP-complete.
\end{theorem}

\begin{theorem} \label{th:cbnnp}
    Let $\Nl$ be a continuous geometric graph, and let $k$ be a positive integer number. The problem of deciding whether  $b(\Nl) <k$ is  NP--complete.
\end{theorem}
\begin{proof}
    Let $\Nl$ be the empty graph on $n$ vertices; since $\Nl$ is non-connected, its diameter is infinity. If we insert a line $\ell$ to ``partition" $\Nl$ into the graphs $\Nl_{\ell}^+$ and $\Nl_{\ell}^-$, the result is that we connect a number of vertices of $\Nl$  with the segment $s_{\ell}$. The graphs $\Nl_{\ell}^+$ and $\Nl_{\ell}^-$ consist of the segment $s_{\ell}$ (duplicated to belong to each graph) and the isolated vertices in, respectively, $\ell^+$ and $\ell^-$, which have not been connected by inserting $\ell$. While any isolated vertices remain after partitioning, the diameter of the corresponding graph will be the same as that of $\Nl$ (infinity). Moreover, each inserted line $\ell$ divides only one of the existing graphs into two graphs of the type $\Nl_{\ell}^+$ and $\Nl_{\ell}^-$, and therefore $b(\Nl)-1$ is the minimum number of lines $\ell$ needed to cover all isolated vertices. In fact, the objects we are using to cover the vertices of $\Nl$ can be seen as a line, half-lines, and segments, where each object is inserted in only one existing region and divides it into two regions. This is a covering by guillotine cuts; thus, $b(\Nl)-1$ is the minimum number of guillotine cuts needed to cover all the points (vertices of $\Nl$.) The result then follows from Theorem~\ref{th:guillotine}.
    \end{proof}
    
\section{Borsuk number of monotone continuous geometric graphs}\label{sec:monotone}

We further explore the problem of how large the Borsuk number of a continuous graph might be by studying this parameter for a large class of continuous geometric graphs, those that are \emph{monotone}. Our main result in this section provides an upper bound on their Borsuk number (Theorem~\ref{thm:k-disjoint}).

Let $\Nl$ be continuous geometric graph, and let $\Nl\cup \mathcal{F}_I$ be the part of the plane determined by the graph itself and all its interior faces. The graph $\Nl$ is said to be \emph{monotone with respect to a line $\ell$} or \emph{$\ell$-monotone} if the intersection of any line perpendicular to $\ell$ with $\Nl\cup \mathcal{F}_I$ is either empty, a single point, or a segment; Figure~\ref{fig:monotone_chain} shows an example. Without loss of generality, throughout this section we assume that the line $\ell$ is horizontal, and thus $\ell^\perp$ is vertical.

For an $\ell$-monotone geometric continuous graph $\Nl$, and a vertical line $\ell^\perp$ that is being translated from left to right (parameterized by the $x$-coordinate), we define the functions: 

$$h^+(x):={\rm diam}\left(\Nl^+_{\ell^\perp(x)}\right) \quad {\rm and} \quad h^-(x):={\rm diam}\left(\Nl^-_{\ell^\perp(x)}\right).$$

\begin{figure}
\centering
	\includegraphics{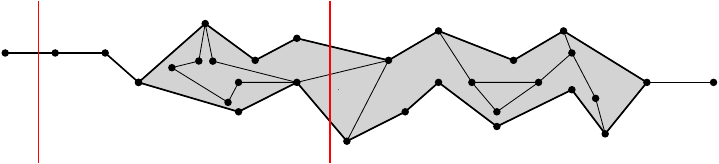}
     \caption{A continuous geometric graph $\Nl$ that is monotone with respect to the $x$-axis; $\Nl\cup \mathcal{F}_I$ consists of the graph (in black) and the gray region. Red vertical lines either intersect $\Nl\cup \mathcal{F}_I$ at a single point or at a segment.}
	\label{fig:monotone_chain}
\end{figure}

\begin{lemma}
\label{cor:contains}
The function $h^-(x)$ (resp., $h^+(x)$) is monotone and increasing (resp., decreasing).
\end{lemma}

\begin{proof}
We argue with $h^-(x)={\rm diam}\left(\Nl^-_{\ell^\perp(x)}\right)$; the argument for 
$h^+(x)$ is analogous.

Let $\ell^\perp(x_1)$ and $\ell^\perp(x_2)$, with $x_1<x_2$, be two vertical lines that intersect $\Nl$. 
As infinite point sets, $(\ell^\perp(x_1)^- \cap \Nl)\subset (\ell^\perp(x_2)^-\cap \Nl)$. Recall that $\Nl^-_{\ell^\perp(x_1)}=(\ell^\perp(x_1)^- \cap \Nl)\cup s_{\ell^\perp(x_1)}$ and $\Nl^-_{\ell^\perp(x_2)}=(\ell^\perp(x_2)^- \cap \Nl)\cup s_{\ell^\perp(x_2)}$. Refer to Figure~\ref{fig:monotone_chain3}. We claim that $h^-(x_1)\leq h^-(x_2)$.

Consider the endpoints $u,v$ of segment $s_{\ell^\perp(x_1)}$ in $\Nl^-_{\ell^\perp(x_1)}$.
When moving the line from $\ell^\perp(x_1)$ to $\ell^\perp(x_2)$,  segment $s_{\ell^\perp(x_1)}$ disappears, but $u$ and $v$ remain connected through a path in $\Nl^-_{\ell^\perp(x_2)}$, whose  length cannot be smaller than that of $s_{\ell^\perp(x_1)}$.
The same happens for any two points of $s_{\ell^\perp(x_1)} \cap \Nl$ that are connected by part of the segment $s_{\ell^\perp(x_1)}$  (instead of the whole segment), see Figure~\ref{fig:monotone_chain3}a.
Thus, the distance in $\Nl^-_{\ell^\perp(x_2)}$  between any two points of $\Nl^-_{\ell^\perp(x_1)} \setminus \{s_{\ell^\perp(x_1)}\}\subset \Nl^-_{\ell^\perp(x_2)}$ is at least their distance in $\Nl^-_{\ell^\perp(x_1)}$ (notice that the shortest paths connecting the points may not use the lines).
In particular, if there exists a diametral pair $p,q$ of $\Nl^-_{\ell^\perp(x_1)}$  located on $\Nl^-_{\ell^\perp(x_1)}\setminus\{s_{\ell^\perp(x_1)}\}$, we have $d_{\Nl^-_{\ell^\perp(x_1)}}(p,q)\leq d_{\Nl^-_{\ell^\perp(x_2)}}(p,q)$. This implies ${\diam}(\Nl^-_{\ell^\perp(x_1)})\leq {\rm diam}(\Nl^-_{\ell^\perp(x_2)})$.

Suppose now that all diametral pairs $p,q$ of $\Nl^-_{\ell^\perp(x_1)}$  satisfy that one of the points, say $q$, is on $s_{\ell^\perp(x_1)}$.
The fact that $\Nl$ is $\ell$-monotone guarantees that $q$ belongs to an interior face of $\Nl^-_{\ell^\perp(x_2)}$. 

If there is a pair $p,q$ such that the point $q$ is not located on the boundary of the outer face of $\Nl$, we can take the horizontal line passing through $q$, which will eventually intersect some vertex or edge at a point $r\in\Nl^-_{\ell^\perp(x_2)}\setminus \Nl^-_{\ell^\perp(x_1)}$; see Figure~\ref{fig:monotone_chain3}b. 
We have  $d_{\Nl^-_{\ell^\perp(x_1)}}(p,q) < d_{\Nl^-_{\ell^\perp(x_2)}}(p,r)$; indeed, any path from $p$ to $r$ must cross $\ell^\perp(x_1)$, and, since $q$ lies on $\ell^\perp(x_1)$, $q$ is closer to that crossing point than $r$. Hence, ${\diam}(\Nl^-_{\ell^\perp(x_1)})\leq {\rm diam}(\Nl^-_{\ell^\perp(x_2)})$. 

Otherwise, for every diametral pair $p, q$ of $\Nl^-_{\ell^\perp(x_1)}$, the point $q$ lies on $s_{\ell^\perp(x_1)}$ and  on the boundary of the outer face of $\Nl$, then the point $r$ may not exist. In that case, one can simply move along the boundary of this outer face to obtain a point at a larger distance than $q$ (from $p$) in $\Nl^-_{\ell^\perp(x_2)} \setminus \Nl^-_{\ell^\perp(x_1)}$. Notice that the monotonicity hypothesis guarantees that this is the only situation in which the point $r$ may not exist.

Finally, note that $p$ and $q$ cannot be both on $s_{\ell^\perp(x_1)}$ since they are diametral. Thus, in every case, we have 
${\diam}(\Nl^-_{\ell^\perp(x_1)})\leq {\rm diam}(\Nl^-_{\ell^\perp(x_2)})$, that is, $h^-(x_1)\leq h^-(x_2)$. 
\end{proof}

\begin{figure}
\centering
	\includegraphics{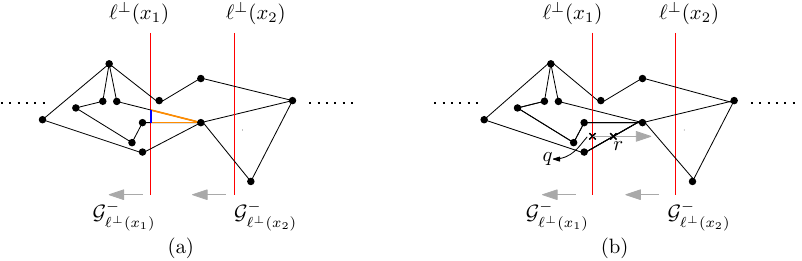}
     \caption{A portion of the monotone graph $\Nl$ in Figure~\ref{fig:monotone_chain}; we compare distances in $\Nl^-_{\ell^\perp(x_1)}$ with distances in $\Nl^-_{\ell^\perp(x_2)}$: (a) the orange path in $\Nl^-_{\ell^\perp(x_2)}$ is longer than the blue segment in $\Nl^-_{\ell^\perp(x_1)}$; (b) finding the point $r$ in $\Nl^-_{\ell^\perp(x_2)}\setminus \Nl^-_{\ell^\perp(x_1)}$.  }
	\label{fig:monotone_chain3}
\end{figure}

Note that, in general, the diameter can increase when partitioning a continuous geometric graph with a line (for example,  partition the graph in Figure~\ref{fig:diameter_shortcut_worse}a with a horizontal line, as in Figure~\ref{fig:diameter_shortcut_worse}b, but moved slightly above the two vertices.) As a straightforward consequence of Lemma~\ref{cor:contains}, we conclude that this cannot happen in graphs that are monotone, which is key to upper bound their Borsuk number.

\begin{corollary}\label{th:convexdiameter}
The associated functions $h^+(x)$ and $h^-(x)$ of an $\ell$-monotone continuous geometric graph $\Nl$ are  upper bounded by $\diam(\Nl)$.
\end{corollary}

To reach the main result in this section, it remains to introduce one final ingredient: the concept of \emph{diametral set}. 

For any continuous geometric graph $\Nl$ (not necessarily monotone), the diametral set $D_{\Nl}(p,q)$ associated to a diametral pair $p,q$ of $\Nl$ is defined as the union of all the shortest $pq$-paths. The set $D_{\Nl}(p,q)$ is a subset of $\Nl$. Notice that, for example, a cycle has an infinite number of diametral pairs of points (see Figure~\ref{fig:infinite_diameters}), but only one distinct diametral set, which is the whole cycle (the union of the two diametral paths for each diametral pair is the same, the cycle). Thus, while a graph can have an infinite number of diametral pairs of points, this is not the case for diametral sets, as we prove in the following lemma.

\begin{lemma} \label{lem:diamsets}
 Let $\Nl$ be a continuous geometric graph with $n$ vertices. The number of distinct diametral sets of $\Nl$ is in $O(n^2)$.
\end{lemma}
\begin{proof}
There always exists a diametral pair of $\Nl$ formed by either: (i) two vertices, (ii) two points on distinct non-pendant edges,\footnote{A \emph{pendant vertex} is a vertex of degree 1. A \emph{pendant edge} is that incident with a pendant vertex.} or (iii) a pendant vertex and a point on a non-pendant edge (see, for instance,~\cite[Lemma 6]{CACERES2018192}). Figure~\ref{fig:cases} illustrates these types of diametral pairs of points of $\Nl$.
Moreover, two distinct diametral pairs lying on the same two edges have the same diametral paths, that is, they define the same diametral set:  indeed, if a point $p$ is part of a diametral pair $p,q$, and lies on a non-pendant edge, then there exist two shortest $pq$-paths, each passing through one of the endpoints of the edge (see~\cite[Lemma 5]{CACERES2018192}). Thus, every diametral pair of points of $\Nl$ can be seen as a pair that may be (i) vertex-vertex, (ii) edge-edge, or (iii) vertex-edge. 
As a consequence, we obtain that there are $O(n^2)$ distinct diametral sets.
\end{proof}

\begin{figure}
\centering
    \includegraphics{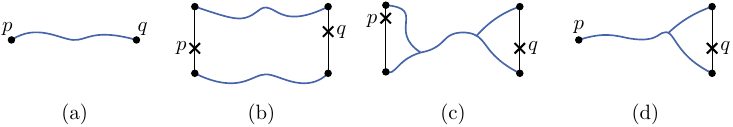}
     \caption{Diametral pairs of points of a continuous graph, where (a) is of type (i), (b) and (c) are of type (ii), and (d) is of type (iii).}
	\label{fig:cases}
\end{figure}

Next, we establish a linear upper bound (in the number of disjoint diametral sets) for the Borsuk number of monotone  graphs.

\begin{theorem}
\label{thm:k-disjoint}
    Let $\Nl$ be an $\ell$-monotone continuous geometric graph, and let $k\geq 1$. 
    If $\Nl$ has at most $k$ disjoint diametral sets, 
    then $b(\Nl) \leq k+2$.
\end{theorem}
\begin{proof}
Let $\mathcal{D}(\Nl)$ be the set of all diametral sets $D_{\Nl}(p,q)$ of $\Nl$, say that there are $k$ distinct diametral pairs $p,q$ in $\Nl$. By Lemma \ref{lem:diamsets}, we have $k=O(n^2)$. Without loss of generality, we assume that no two vertices of $\Nl$ have the same $y$-coordinate. 

By Corollary \ref{th:convexdiameter}, to reduce the original diameter when cutting $\Nl$ by lines, it suffices to intersect all diametral sets in $\mathcal{D}(\Nl)$ with lines perpendicular to $\ell$, as the new points on the cutting lines cannot cause an increase of $\diam(\Nl)$.  Moreover,  for each diametral set $D_{\Nl}(p,q)$, it suffices to intersect a single diametral path between $p$ and $q$, since this decreases the distance between $p$ and $q$, and hence they are no longer diametral.

We claim that all sets can be intersected using $k+1$ lines. Indeed, project each diametral set $D_{\Nl}(p,q)$ onto the line $\ell$ so that each set determines an interval on the line. Then, sort the intervals by their rightmost extreme. We denote them by $[p_i,q_i]$, $1\leq i\leq k$, where $q_j<q_t$ for $j<t$. Set $\varepsilon$ smaller than half the distance between any two consecutive points in the set $\{p_i,q_i \, | \, 1\leq i\leq k\}$. The line ${\ell_1^\perp}: x=q_1-\varepsilon$ crosses the first interval $[p_1,q_1]$ and all the intervals that intersect with that. The intervals that are not crossed by ${\ell_1^\perp}$ are disjoint with $[p_1,q_1]$, which means that the diametral sets that gave rise to the corresponding intervals cannot intersect.

We now discard from $\mathcal{D}(\Nl)$ all the elements crossed by ${\ell_1^\perp}$ to obtain a subset $\mathcal{D}_1(\Nl)$ of $\mathcal{D}(\Nl)$, and repeat the same process to obtain a new line ${\ell_2^\perp}$ and a new subset $\mathcal{D}_2(\Nl)$ of $\mathcal{D}_1(\Nl)$. This produces a sequence of subsets of diametral sets $\mathcal{D}_0(\Nl)=\mathcal{D}(\Nl)\supset \mathcal{D}_1(\Nl) \supset \mathcal{D}_2(\Nl) \ldots$ satisfying that the diametral sets in $\mathcal{D}_{i-2}(\Nl)\backslash \mathcal{D}_{i-1}(\Nl)$ do not intersect those in $\mathcal{D}_{i-1}(\Nl)\backslash \mathcal{D}_{i}(\Nl)$. Hence, we can find at most $k+1$ of the lines defined above, otherwise we would have $k+1$ disjoint diametral sets. These $k+1$ lines split $\Nl$ in, at most, $k+2$ parts, each with diameter smaller than $\Nl$, since they intersect all diametral sets of $\Nl$. Therefore, $b(\Nl) \leq k+2$.
\end{proof}

We note that the previous bound can be attained, at least for $k=1$.

 \begin{observation}
Let $\mathcal{W}_{33}$ be the wheel graph on 33 vertices embedded in the plane such that its outer boundary is a regular 32-sided polygon, and the distance from the wheel center to each polygon vertex is one (see Figure~\ref{fig:wheel}.) Then, $b(\mathcal{W}_{33}) = 3$.\end{observation}

\begin{proof}
Each side of the outer boundary has length $\omega=2 \sin (\pi/32) \approx 0.19$. Any two vertices of the polygon are connected by a path of length two, through the wheel center.
This path is shorter than going along the boundary as soon as the other vertex is more than ten vertices away along the boundary (since $11\omega>2$).
It follows that the diametral pairs of this continuous graph are given by pairs of midpoints of polygon sides that are at distance $2 + \omega \approx 2.19$.
In fact, each midpoint has nine points at exactly that distance, corresponding to the midpoint exactly opposite, plus those of the first four sides neighboring the opposite side, in each direction. Figure~\ref{fig:wheel} shows the nine diametral pairs involving one of the midpoints.

\begin{figure}
  \begin{center}
    \includegraphics{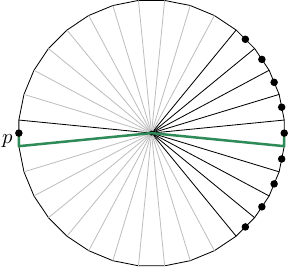}
  \end{center}
  \caption{The wheel graph on 33 vertices attains the bound of Theorem \ref{thm:k-disjoint} for $k=1$; the black thick points indicate the nine diametral pairs involving a point $p$.}
  \label{fig:wheel}
\end{figure}

Next we argue that subdividing by one line is not enough to decrease the diameter of $\mathcal{W}_{33}$.
Any line intersecting the wheel will leave at least 15 complete triangles of the wheel on one side.
These triangles are contiguous, and form a fan. 
The diameter of any such a fan with 13 or more triangles remains the same as the original one, $2+\omega$. 

It follows that two lines are necessary. Moreover, they are also sufficient, since two parallel lines at a very small distance that enclose the center will result in three subgraphs with smaller diameter.
Therefore, $b(\mathcal{W}_{33})=3=k+2$.
\end{proof}

We conclude this section by delving deeper into the two key tools we have used: the functions $h^+(x)$ and $ h^-(x)$, and the number of disjoint diametral sets that a continuous geometric graph $\Nl$ can have. We think that this additional information, which is of independent interest, can contribute to progress on the question of how large $b(\Nl)$ can be.

\subsection{Convex monotone graphs}

An $\ell$-monotone continuous geometric graph $\Nl$ is said to be \emph{convex with respect to the line $\ell$} or \emph{$\ell$-convex} if all its interior faces are convex with respect to $\ell$ (i.e., every line perpendicular to $\ell$ that intersects the face does it either at a point or at a segment).

\begin{proposition}
\label{lem:seis}
The associated functions $h^+(x)$ and $h^-(x)$ of an $\ell$-convex graph $\Nl$  are continuous.
\end{proposition}
\begin{proof}
Consider the $n-1$ strips determined by lines perpendicular  to $\ell$ through each vertex of $\Nl$; we denote them by $\ell^\perp(x_i)$ with $1\leq i\leq n=|V(\Nl)|$, and $x_i<x_j$ for $i<j$. Refer to Figure~\ref{fig:continuity}. (Again, without loss of generality, we assume that no two vertices of $\Nl$ have the same $y$-coordinate.) 

We claim that, for every $1\leq i\leq n$, the function $h^-(x)$ is continuous for $x\in [x_i, x_{i+1}]$ (informally, when the perpendicular line to $\ell$ moves inside any of the strips), and this is true even without the convexity hypothesis.

Indeed, consider an interval $[x_i,x_{i+1}]$ (which determines a strip), a line $\ell^\perp(x_j)$ with $x_j\in [x_i,x_{i+1}]$, and the line $\ell^\perp(x_j-\delta)$ for a sufficiently small value $\delta \in (0, x_j-x_i)$; see Figure~\ref{fig:continuity}a.
 Let $p,q$ be diametral points of $\Nl^-_{\ell^\perp(x_j)}=(\ell^\perp(x_j) \, ^-\cap \Nl)\cup s_{\ell^\perp(x_j)}$. Suppose first that neither $p$ nor $q$ are on $s_{\ell^\perp(x_j)}$, and choose $\delta$ so that they are on the left of $\ell^\perp(x_j-\delta)$. We distinguish two cases.

\,

\noindent \emph{Case 1. No diametral $pq$-path contains part of the segment $s_{\ell^\perp(x_j)}$}. Then, no diametral $pq$-path intersects $\ell^\perp(x_j-\delta)$, as inside a strip there are no vertices to cross the line and go back. Thus, line $\ell^\perp(x_j-\delta)$ does not shorten any diametral $pq$-path. Hence, $$h^-(x_j)=d_{\Nl^-_{\ell^\perp(x_j)}}(p,q)=d_{\Nl^-_{\ell^\perp(x_j-\delta)}}(p,q)\leq h^-(x_j-\delta).$$ Further,
by Lemma \ref{cor:contains}, $ h^-(x_j-\delta)\leq h^-(x_j)$, which implies that $ h^-(x_j)=h^-(x_j-\delta)$.

\

\noindent \emph{Case 2. There is a diametral $pq$-path that contains part of segment $s_{\ell^\perp(x_j)}$.}
The situation in this case is illustrated in Figure~\ref{fig:continuity}b. Let $r_j$ and $t_j$ be the endpoints of the sub-segment of $s_{\ell^\perp(x_j)}$ contained in the diametral path, and let $r_{j-\delta}$ and $t_{j-\delta}$ be the intersection points of the path with the line $\ell^\perp(x_j-\delta)$. As there are no vertices in the strip, points $r_j$ and $r_{j-\delta}$ must belong to the same edge of $\Nl$, and the same happens with $t_j$ and $t_{j-\delta}$. This means that, in the strip determined by $[x_i,x_{i+1}]$, line $\ell^\perp(x)$ moves linearly along two edges, which is a continuous movement. Hence, $d_{\Nl^-_{\ell^\perp(x_j-\delta)}}(p,q)=d_{\Nl^-_{\ell^\perp(x_j)}}(p,q)-\varepsilon$ for $\varepsilon>0$. By Lemma \ref{cor:contains},  $$d_{\Nl^-_{\ell^\perp(x_j)}}(p,q)-\varepsilon=d_{\Nl^-_{\ell^\perp(x_j-\delta)}}(p,q)\leq h^-(x_j-\delta)\leq h^-(x_j)=d_{\Nl^-_{\ell^\perp(x_j)}}(p,q).$$
When $\varepsilon$ tends to zero, $h^-(x_j)-h^-(x_j-\delta)$ also tends to zero.

\
 
The argument is similar when either $p$ or $q$ are on $s_{\ell^\perp(x_j)}$, say $q$. If point $q$ is on an edge of $\Nl$, we take the point $r_{j-\delta}$ (again, the intersection point of the $pq$-path with the line $\ell^\perp(x_j-\delta)$) and argue as in case 2 above, where instead of two edges given by points $r_j,r_{j-\delta},t_j,t_{j-\delta}$, there is only one edge given by $r_{j-\delta}$ and $q$. Now, if $q$ is not a point in $\Nl$, we consider the region bounded by the lines $\ell^\perp(x_j),\ell^\perp(x_j-\delta)$ and the two edges enclosing $q$. Project $q$ (in the direction of $\ell$) on the boundary of that region obtaining the point $q'$, see Figure~\ref{fig:continuity}c. We argue as in case 2, but taking $d_{\Nl^-_{\ell^\perp(x_j-\delta)}}(p,q')=d_{\Nl^-_{\ell^\perp(x_j)}}(p,q)-\varepsilon$.

The convexity hypothesis lets us conclude that our argument proves the continuity in each closed strip, and so we can extend the continuity to the left and the right ``moving'' from a strip to its (at most) two neighboring strips. Note that if $\Nl$ is convex, every vertex has incident edges in the (at most) two strips it belongs to.
\end{proof} 

\begin{figure}
\centering
	\includegraphics{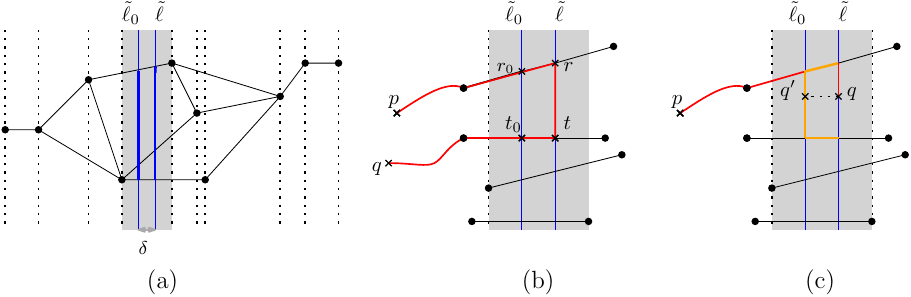}
     \caption{(a) A convex graph, which is split into continuous subgraphs, each contained in one strip;  lines $\ell^\perp(x_j)$ and $\ell^\perp(x_j-\delta)$ in blue; line $\ell^\perp(x_j-\delta)$ is indicated by the value $\delta$, and segments $s_{\ell^\perp(x_j)}$ and $s_{\ell^\perp(x_j-\delta)}$ are indicated in thicker blue. (b) Diametral points $p,q$ that are not on $s_{\ell^\perp(x_j)}$, and a diametral $pq$-path (in red) that contains part of $s_{\ell^\perp(x_j)}$. (c) Point $q$ is now on $s_{\ell^\perp(x_j)}$ but not in $\Nl$, and point $q'$ is the projection (in the direction of convexity) of $q$ onto the orange region.}
	\label{fig:continuity}
\end{figure}

\subsection{Disjoint diametral sets}

The upper bound of Theorem~\ref{thm:k-disjoint} depends on the number of disjoint diametral sets that a monotone graph  has. This raises the question of how many disjoint diametral sets or, in particular, disjoint diametral paths a continuous geometric graph $\Nl$ (not necessarily monotone) might admit, taking into account the planarity constraint. This problem is far from being solved; our aim in this section is to provide a first insight into the structure of the graphs that contain disjoint diametral sets.
As an example, Figure~\ref{fig:3_disjoint_diameters} illustrates a graph with three disjoint diametral paths. Notice that, in this graph, there are also diametral paths that are not disjoint (each pair of square points is a diametral pair.)

\begin{figure}[htbp]
\begin{center}
\includegraphics{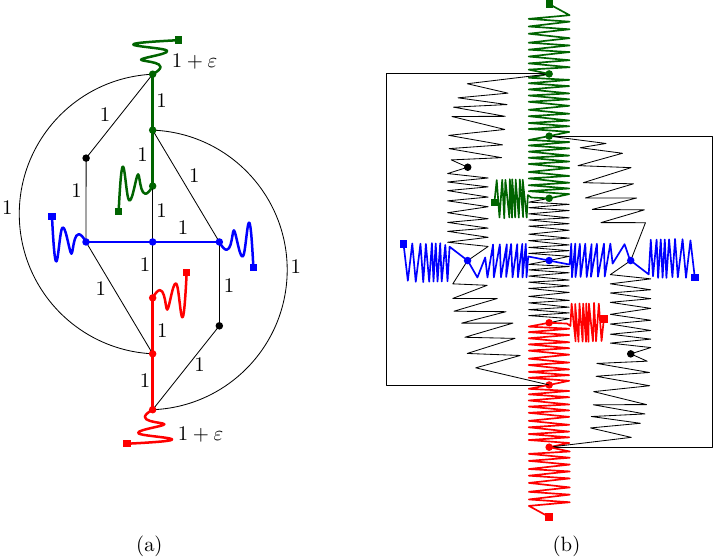}
\end{center}
\caption{(a) Sketch of a  graph with three disjoint diametral paths (green, red, and blue); winding edges have length $1+\varepsilon$, all others have length $1$, resulting in a diameter of $4+2\varepsilon$; (b) an equivalent polygonal (approximate) geometric realization of the graph.}\label{fig:3_disjoint_diameters}
\end{figure}

For two fixed pairs $p,q$ and $p',q'$ of diametral points of  $\Nl$, consider any subgraph of $\Nl$ that contains six shortest $ij$-paths with different endpoints $i,j \in\{p,q,p',q'\}$, $i\neq j$. We use $\Nl(pq,p'q')$ to denote this selected subgraph; see Figure~\ref{fig:impossible3}. 
By testing all possible combinatorial ways to have two disjoint diametral sets, one can verify that the only way in which this is possible is the four diametral points are not on the same face of \Nl.

\begin{proposition} \label{prop:disjdiam}
Let $p,q$ and $p',q'$ be two pairs of diametral points of a continuous geometric graph $\Nl$. If the diametral sets $D_{\Nl}(p,q)$ and $D_{\Nl}(p',q')$ are disjoint then, for any planar embedding of $\Nl(pq,p'q')$, the four points $\{p,q,p',q'\}$ are not located in the same face. 
\end{proposition}

\begin{proof}[Proof sketch]
    Our proof is computational.
    We consider two diametral paths: the one between $p$ and $q$, and the one between $p'$ and $q'$. 
    They must be disjoint, but there need to exist four \emph{crossing paths} between: $p$ and $p'$, $p$ and $q$, $q$ and $p'$, and $q$ and $q'$. 
    Each of these paths can use part of the two diametral paths. 
    We model this by subdividing each diametral path into five sections, to account for all the possible ways in which the four crossing paths can share parts of the two shortest paths.
    See Figure~\ref{fig:program}.
    Consider, for example, the shortest path between $p$ and $p'$: this path will start following the shortest path from $p$ to $q$ until certain point (possibly equal to $p$)---in the figure, the point labeled 0, then connect to the shortest path from $p'$ and $q'$ at some other point---in the figure, point 4, and then will follow it until reaching $p'$.
    In our program, we discretize the points at which the crossing paths leave/join the $p$-$q$ and $p'$-$q'$ shortest paths.
    To account for all possible situation, it is enough to have four points on each shortest path.
    In the figure, these are labeled $0$--$3$ and $4$--$7$.
    This means that each of the four crossing paths has $4^2$ combinatorial options, leading to $(4^2)^4=65536$ total cases.

    We wrote a program in Sagemath that tries each of these cases.
    For each case, we determine if there is an assignment of path lengths that results in the desired configuration: one where $p,q$ and $p',q'$ are diametral pairs.
    We determine this by formulating the problem as a linear program, where the variables correspond mainly to the lengths of the five subpaths on each diametral path (e.g., variable $x_{01}$ in the figure), and four variables for the parts of the crossing paths connecting between the two shortest paths (e.g., variable $x_{pp'}$ in the figure).

    In case a combinatorial case is feasible, we check if vertices $p,q,p',q'$ can be placed on the same face. 
    To do this, we add a new vertex $v_\infty$ to the graph, with edges to $p,q,p',q'$, and check if the resulting graph is planar.
    In all cases, the resulting graph turned out to be non-planar, thus implying that the only possible way in which the two diametral paths can be disjoint is by not having all four of $p,q,p',q'$ on the same face.
    
    The complete Sagemath code is provided in the appendix.
\end{proof}

\begin{figure}[tb]
    \centering
    \includegraphics{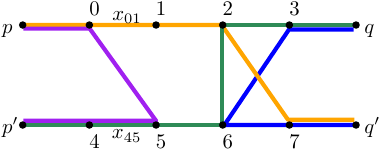}
    \caption{Example with concrete entry points for each of the four crossing paths on the two shortest paths.}
    \label{fig:program}
\end{figure}

Figure \ref {fig:impossible3} shows that the converse of Proposition \ref{prop:disjdiam} is not true in general. Indeed, in this example, the diametral sets $D_{\Nl}(p,q)$ and $D_{\Nl}(p',q')$ consist of  single shortest paths (connecting the corresponding diametral points); the paths intersect at one point, denoted by $z$. On the other hand, one can prove that there cannot be a planar embedding of $\Nl(pq,p'q')$ with the four points $\{p,q,p',q'\}$  located in the same face. This comes from the fact that this graph is not outerplanar (it contains a bipartite graph $K_{2,3}$ as an induced subgraph.)

\begin{figure}[htbp]
    \centering
    \includegraphics[width=5cm]{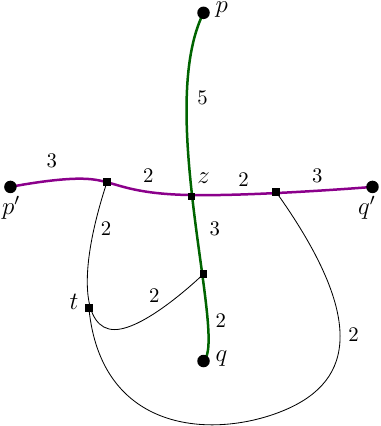}
    \caption{Sketch of $\Nl(pq,p'q')$. The shortest $pq$-path and $p'q'$-path intersect at point $z$. The square points induce a bipartite graph $K_{2,3}$, where  $\{t,z\}$ is the class of size $2$.}
    \label{fig:impossible3}
\end{figure}


\section{Borsuk number of trees}\label{sec:trees}

In Observation~\ref{obs:examples} we exhibit examples of  trees where the Borsuk number, in the discrete case, can be easily computed, showing in particular that this parameter can be linear in the number of vertices. In this section, we first compute the Borsuk number of an arbitrary tree $T$, and then we move to the continuous version of the problem, which behaves differently.

\begin{proposition}
The Borsuk number of any tree $T$ with $n$ vertices can be computed in $O(n)$ time. Furthermore, 
\begin{enumerate}
\item[(i)] If the center of $T$ is not a unique vertex, then $b(T)=2$.
\item[(ii)] If the center of $T$ is a unique vertex $v$, then $b(T)=b(T')=\delta_{T'}(v)$, where $T'$ is the subtree of $T$ induced by the vertices of all diametral paths, and $\delta_{T'}(v)$ is the degree of $v$ in $T'$.
\end{enumerate}
\end{proposition}
\begin{proof}
If the center of  $T$ consists of two adjacent vertices, it suffices to remove the edge incident to both vertices  to obtain two trees with smaller diameter (since all diametral paths contain that edge); this gives $b(T)=2$.
When the center is a single vertex $v$, we need to delete all edges incident to $v$ that are in a diametral path except one, that is, $\delta_{T'}(v)-1$ edges, which gives rise to $\delta_{T'}(v)$ connected components, and so $b(T)=b(T')=\delta_{T'}(v)$.

Since the center of a (weighted) tree can be computed in linear time~\cite{HCH81}, so does its Borsuk number.
\end{proof}

 While the value of $b(T)$ depends on the center of the tree $T$, next we show that the Borsuk number of  a continuous geometric tree $\T$ is upper bounded by a constant. We begin with a technical lemma, which establishes that lines intersecting a tree  at its center cannot cause an increase of the diameter of the original tree.

  \begin{lemma}\label{lem:cycles}
Let $\T$ be a continuous geometric tree with center point $\mathcal{C}_{\T}$, and let  $\ell$ be a line that passes through $\mathcal{C}_{\T}$. Then,
\begin{equation}\label{eq:non-stric}
\max\{{\rm diam}(\T_{\ell}^+), {\rm diam}(\T_{\ell}^-)\}\leq {\rm diam}(\T),\end{equation}
 \end{lemma}

 \begin{proof}
We prove the lemma for $\T_{\ell}^+$; the arguments for $\T_{\ell}^-$ are equivalent.
Recall that $\T_{\ell}^+=(\ell^+\cap \T)\cup s_{\ell}$.
Since, by definition, the center point $\mathcal{C}_{\T}$ is the midpoint of all diametral paths of $\T$, it suffices to show that $d_{\T_{\ell}^+}(p,\mathcal{C}_{\T})\leq \diam(\T)/2$ for every point $p\in \T_{\ell}^+$
(the distance between two diametral points would then be at most $\diam(\T)$).
This upper bound for $d_{\T_{\ell}^+}(p,\mathcal{C}_{\T})$ is clear when $p\in \T$, since $\mathcal{C}_{\T}$ is the center of $\T$. Now, suppose that $p\in s_{\ell} \setminus\{\T\}$. Starting at $\mathcal{C}_{\T}$, walk through segment $s_{\ell}$ (in the direction to $p$), and consider the first point $q\in s_{\ell}\cap \T$ that is located after $p$. Since, in $\T_{\ell}^+$, the distances of $p$ and $q$ to $\mathcal{C}_{\T}$ are given by the lengths of the straight-line segments $\overline{p\mathcal{C}_{\T}}$ and $\overline{q\mathcal{C}_{\T}}$, we have
$$d_{\T_{\ell}^+}(p,\mathcal{C}_{\T})<d_{\T_{\ell}^+}(q,\mathcal{C}_{\T})<d_{\T}(q,\mathcal{C}_{\T})\leq \frac{\diam(\T)}{2}.$$
 \end{proof}

\begin{remark}\label{obs:non-center}
Lemma~\ref{lem:cycles} also holds for lines that intersect the tree, not at the center, but infinitely close to it. The difference in the proof would be that $d_{\T_{\ell}^+}(p,\mathcal{C}_{\T})$ and $d_{\T_{\ell}^+}(q,\mathcal{C}_{\T})$ are not given exactly by the lengths of the segments $\overline{p\mathcal{C}_{\T}}$ and $\overline{q\mathcal{C}_{\T}}$, but we must add a term that tends to zero (the distance from the closest point in $\mathcal{C}_{\T}\cap \ell$ to $\mathcal{C}_{\T}$.)
\end{remark}

\begin{remark}\label{rmk:half-diameter}
 In the proof of Lemma~\ref{lem:cycles}, we show that if  $\T$ is cut by a line that goes through its center $\mathcal{C}_{\T}$, the distance from  $\mathcal{C}_{\T}$ to any point on the resulting graphs is at most $\diam(\T)/2$, and strictly smaller when the point is on the line but not on the tree. Using Remark~\ref{obs:non-center}, this can be extended to lines that intersect the tree infinitely close to the center.
\end{remark}

\begin{proposition}
\label{prop:tree}
If $\T$ is continuous geometric tree, then $b(\T)\leq 3$.
\end{proposition}

\begin{proof}

Let $\ell$ be a line that passes through the center point $\mathcal{C}_{\T}$ of $\T$, and divides the tree into two graphs, $\T_{\ell}^+=(\ell^+\cap \T)\cup s_{\ell}$ and $\T_{\ell}^-=(\ell^-\cap \T)\cup s_{\ell}$. Assume, without loss of generality, that  $\ell$ does not contain any edge incident or passing through $\mathcal{C}_{\T}$. 
By Lemma~\ref{lem:cycles}, the diameters of $\T_{\ell}^+$ and $\T_{\ell}^+$ satisfy~(\ref{eq:non-stric}) and, moreover, by Remark~\ref{rmk:half-diameter}, we have $d_{\T_{\ell}^+}(p,\mathcal{C}_{\T})\leq \diam(\T)/2$ for every point $p\in \T_{\ell}^+$ (the same for $\T_{\ell}^-$.) We distinguish two cases.

\

\noindent \emph{Case 1. The center $\mathcal{C}_{\T}$ is not a vertex of $\T$.}

\noindent  We prove that, in this case, $\diam(\T_{\ell}^-)<\diam(\T)$; the argument is analogous for $\T_{\ell}^+$. This implies that  $\max\{\diam(\T_{\ell}^+),\diam(\T_{\ell}^-)\}< \diam(\T)$, which gives $b(\T)=2$.
Note that the center $\mathcal{C}_{\T}$ is located in the interior of an edge $uv$ with, say, $u\in \T_{\ell}^-$ and $v\in \T_{\ell}^+$.

Let $p,q\in \T_{\ell}^-$, and suppose first that $p,q\in\T \setminus \{s_{\ell}\cap \T\}$. 

If one of the points is not a leaf, then $d_{\T_{\ell}^-}(p,q)< \diam(\T)$, since the distance in $\T$ between the points before inserting  $\ell$ was already smaller than the diameter of $\T$. 

If $p$ and $q$ are leaves of $\T$,  see Figure~\ref{fig:tree}a, we have
$$d_{\T_{\ell}^-}(p,q)<d_{\T_{\ell}^-}(p,\mathcal{C}_{\T})+d_{\T_{\ell}^-}(q,\mathcal{C}_{\T})\leq \frac{\diam(\T)}{2}+\frac{\diam(\T)}{2}=\diam(\T),$$
since a shortest $pq$-path contains the vertex $u$, but not the whole edge $uv$ where $\mathcal{C}_{\T}$ is located.

Assume now that only one of the points is on the tree, say, $p\in \T$ and $q\in s_{\ell} \setminus \T \cap s_{\ell}$; see Figure~\ref{fig:tree}b.  It could happen that $d_{\T_{\ell}^-}(p,\mathcal{C}_{\T})=\diam(\T)/2$ if $p$ were a leaf of a diametral pair of $\T$; however, $d_{\T_{\ell}^-}(q,\mathcal{C}_{\T})=|\overline{q\mathcal{C}_{\T}}|<\diam(\T)/2$, and therefore
$$d_{\T_{\ell}^-}(p,q)\leq d_{\T_{\ell}^-}(p,\mathcal{C}_{\T})+d_{\T_{\ell}^-}(q,\mathcal{C}_{\T})<d_{\T_{\ell}^-}(p,\mathcal{C}_{\T})+\frac{\diam(\T)}{2}\leq \frac{\diam(\T)}{2}+\frac{\diam(\T)}{2}=\diam(\T).$$

Finally, if $p,q\in s_{\ell}\cap \T$ we have $d_{\T_{\ell}^-}(p,q)=|\overline{pq}|< \diam(\T)$. 

\begin{figure}
\centering
	\includegraphics{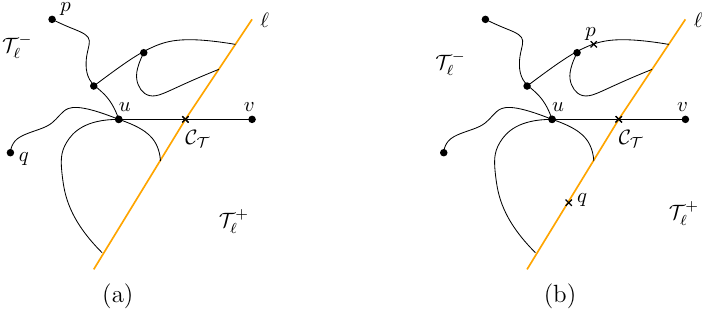}
     \caption{(a) Points $p,q$ are leaves of $\T$, and a shortest $pq$-path goes through $u$. (b) Point $p$ is on $\T$, and $q$ is on segment $s_{\ell}$ but not on the tree; its distance to the center in $\T_{\ell}^-$ is given by the length of segment $\overline{q\mathcal{C}_{\T}}$. }
	\label{fig:tree}
\end{figure}

\

\noindent \emph{Case 2. The center $\mathcal{C}_{\T}$ is a vertex of $\T$.}

\noindent A diametral pair of $\T_{\ell}^-$ (or $\T_{\ell}^+$) may consist of two leaves of $\T$ at distance ${\rm diam}(\T)$. Hence,
 $d_{\T_{\ell}^-}(p,q)\leq d_{\T_{\ell}^-}(p,\mathcal{C}_{\T})+d_{\T_{\ell}^-}(q,\mathcal{C}_{\T})\leq \diam(\T)$. For example, this happens in a star graph with all edges of the same length and not all contained in any half-plane determined by a line through the center. We solve this issue by taking two parallel lines to $\ell$, one slightly above and  the other below. Therefore,  $b(\mathcal{T}) \leq 3$.
\end{proof}

We now turn to the problem of computing the Borsuk number of a continuous geometric tree $\T$. By Proposition~\ref{prop:tree}, this is the same as distinguishing between $b(\T)$ equal to $2$ or $3$. Thus, we first analyze which lines give a correct partition in trees with Borsuk number $2$.

With lines that go exactly through the center $\mathcal{C}_{\T}$, we cannot guarantee that the diameters obtained after cutting are strictly smaller than ${\rm diam}(\mathcal{T})$ (for example, the star graph with three edges of the same length and not contained in the same half-plane through the center.)
However, Lemma~\ref{lem:throughcenter} below states that when a tree has Borsuk number $2$, we can always find a line intersecting $\T$ at a point infinitely close to the center giving a correct partition. 

 \begin{lemma}\label{lem:throughcenter}
Let $\T$ be a continuous geometric tree with center point $\mathcal{C}_{\T}$.  If $b(\T)=2$, then  there exists a sequence of lines  $\{\ell_i\}_{i\geq 0}$ satisfying that:
\begin{enumerate}
\item[(i)]  $\{d_{\T}(t_i,\mathcal{C}_{\T})\}_{i\geq 0}$ approaches zero, where $t_i$ is the closest point in $\T\cap \ell_i$ to $\mathcal{C}_{\T}$.
\item[(ii)] there exists $j\geq 0$ such that for every $i\geq j$, line $\ell_i$ splits $\T$ into two continuous geometric graphs with  diameters smaller than that of  $\T$, that is, 
$$\max\{{\rm diam}(\T_{\ell_i}^+), {\rm diam}(\T_{\ell_i}^-)\}< {\rm diam}(\T).$$
\end{enumerate}
 \end{lemma}
\begin{proof}
Since $b(\T)=2$, there is a line $\ell$ that divides $\T$ into two graphs $\T_{\ell}^+$ and $\T_{\ell}^-$ such that $\max\{\diam(\T_{\ell}^+),\diam(\T_{\ell}^+)\}< \diam(\T)$. 
 If line $\ell$ goes through $\mathcal{C}_{\T}$, the result follows by taking the constant sequence $\ell_i:=\ell$ for every $i$.

Suppose now that $\ell$ does not pass through $\mathcal{C}_{\T}$.
Without loss of generality, we assume that $\ell$ intersects $\T$ in at least two points, otherwise we can always translate it towards $\mathcal{C}_{\T}$ until that happens, or until it passes through $\mathcal{C}_{\T}$ (corresponding to the previous case).
Let $t$ be a closest point to $\mathcal{C}_{\T}$ in $\T\cap\ell$. 
We generate a sequence of lines $\{\ell_i\}_{i\geq 0}$ and the corresponding graphs $\T_{\ell_i}^+$ and $\T_{\ell_i}^-$ by rotating $\ell$ around the farthest point from $\mathcal{C}_{\T}$ in $\T\cap \ell$.
Each line $\ell_i$ generates a point $t_i$, defined as closest point to $\mathcal{C}_{\T}$ in $\T\cap\ell_i$. 
This sequence of lines and points satisfies item (i);  next we show that it also satisfies item (ii).

 By Remark~\ref{obs:non-center}, when lines are sufficiently close to the center, Lemma~\ref{lem:cycles} can be applied. Thus, there exists $j\geq 0$ such that for every $i\geq j$, we have
 \begin{equation}\label{eq:inequality}
\max\{{\rm diam}(\T_{\ell_i}^+), {\rm diam}(\T_{\ell_i}^-)\}\leq {\rm diam}(\T)\end{equation}
 Further, by Remark~\ref{rmk:half-diameter}, the distance from $\mathcal{C}_{\T}$ to any point on the maximal segment $s_{\ell_i}$ that is not on $\T$ is strictly smaller than $\diam(\T)/2$, and so the equality in~(\ref{eq:inequality}) can only be given by diametral pairs of $\T$.
 Now, the graphs $\T_{\ell}^+$ and $\T_{\ell}^-$ do not contain diametral paths of $\T$, all of them must go through $\mathcal{C}_{\T}$ and connect a leaf of $\T_{\ell}^+$ with a leaf of $\T_{\ell}^-$. When $\ell$ is rotated giving rise to the sequence $\{\ell_i\}_{i\geq 0}$, we obtain the graphs $\T_{\ell_i}^+$ and $\T_{\ell_i}^-$ but without going through the center $\mathcal{C}_{\T}$, which always remains in the same half-plane. Hence, there cannot be diametral paths of $\T$ in either $\T_{\ell_i}^+$ or $\T_{\ell_i}^-$, which implies that the inequality in~(\ref{eq:inequality}) is strict.
\end{proof}

We are now ready to prove our main result in this section.

\begin{theorem}
\label{th:tree}
The Borsuk number of any continuous geometric tree $\T$ with $n$ vertices can be computed in $O(n^2)$ time.
\end{theorem}

\begin{proof}
First, we compute the center $\mathcal{C}_{\T}$ of $\T$ in linear time \cite{HCH81}. By  Proposition~\ref{prop:tree}, if $\mathcal{C}_{\T}$ is not a vertex then $b(\T)=2$; otherwise, $b(\T)\leq 3$, and we must characterize when the value is $2$. 
To do this, Lemma~\ref{lem:throughcenter} is key since it tells us that we can focus on lines that are close enough to $\mathcal{C}_{\T}$ as our set of candidate lines to give a correct partition of $\T$.

Let $\ell$ be a line that passes exactly through the center $\mathcal{C}_{\T}$, which is assumed to be a vertex of $\T$ (case 2 in Proposition~\ref{prop:tree}), and divides the tree into $\T_{\ell}^+$ and $\T_{\ell}^-$; again, we assume that $\ell$ does not contain any edge incident with $\mathcal{C}_{\T}$.  By Lemma \ref{lem:cycles} and Remark \ref{rmk:half-diameter}, to check whether line $\ell$ gives a correct partition, that is, whether inequality (\ref{eq:non-stric}) is strict, it suffices to check if the distance between diametral vertices of $\T$ (which are always leaves) decreases when they are both in either $\T_{\ell}^+$ or $\T_{\ell}^-$. 

First we observe that there are $O(n)$ classes of combinatorially equivalent lines that pass through $\mathcal{C}_{\T}$ (i.e., lines through the center intersecting the same subset of edges of $\T$). Indeed, the center $\mathcal{C}_{\T}$ is a vertex, and the other $n-1$ vertices can be ordered angularly around  $\mathcal{C}_{\T}$. This order determines up to $2(n-1)$ double wedges with  apex $\mathcal{C}_{\T}$, defined by pairs of angularly consecutive lines that pass through $\mathcal{C}_{\T}$ and a vertex of $\T$.
Each double wedge is identified with a class of combinatorially equivalent lines; see Figure \ref{fig:wedges2}.

\begin{figure}
\centering
\includegraphics{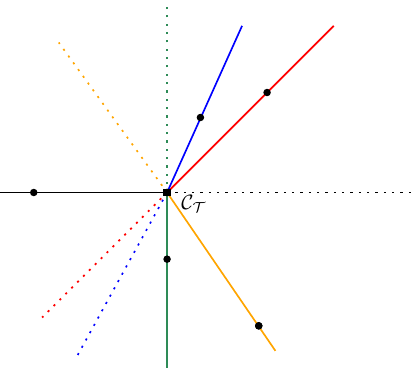}
     \caption{Double wedges generated by the angular order around $\mathcal{C}_{\T}$.}
	\label{fig:wedges2}
\end{figure}

Let  $u,v$ be a diametral pair of leaves   in $\T$, say, $u,v\in \T_{\ell}^-$ (we also use the expression diametral leaf to refer to a leaf that is part of a diametral pair of leaves). We consider the set $P(v,\mathcal{C}_{\T})$  of all vertices of the $v\mathcal{C}_{\T}$-path except $\mathcal{C}_{\T}$, and denote by $\T(v, \mathcal{C}_{\T})$ the subtree of $\T$ induced by the vertices whose path to $\mathcal{C}_{\T}$ has some vertex in $P(v,\mathcal{C}_{\T})$; see Figure~\ref{fig:cases_tree}. 
The analogous set and subtree are also defined for $u$, and denoted by $P(u,\mathcal{C}_{\T})$ and $\T(u, \mathcal{C}_{\T})$, respectively.
We distinguish three cases. Refer to Figure~\ref{fig:cases_tree}.

\begin{figure}
\centering
	\includegraphics{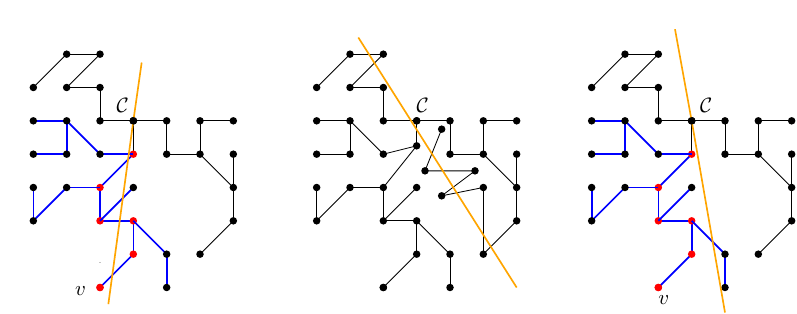}
     \caption{The set $P(v,\mathcal{C}_{\T})$ consists of the red vertices; the blue edges determine the subtree $\T(v, \mathcal{C}_{\T})$. From left to right: cases 1--3 of Theorem \ref{th:tree}.}
	\label{fig:cases_tree}
\end{figure}

\

\noindent \emph{Case 1.} \emph{There are vertices of $P(v,\mathcal{C}_{\T})$ in  both halfplanes determined by $\ell$}. The  path in $\T$ connecting $v$ with $\mathcal{C}_{\T}$ is shortened by a straight-line segment contained in $\ell$. Hence, $d_{\T_{\ell}^-}(v,\mathcal{C}_{\T})<d_{\T_{\ell}}(v,\mathcal{C}_{\T})=\diam(\T)/2$, which implies that $d_{\T_{\ell}^-}(u,v)<\diam(\T)$.

\

\noindent \emph{Case 2.}  \emph{All vertices of $P(v,\mathcal{C}_{\T})$ are in the same halfplane  determined by $\ell$, and line $\ell$ does not intersect the convex hull of $\T(v, \mathcal{C}_{\T})$}. We could argue as in case 1 for the set $P(u,\mathcal{C}_{\T})$ but it might be the case that this set were in the same conditions as $P(v,\mathcal{C}_{\T})$. However, we can always consider two lines combinatorially equivalent  and parallel to $ \ell$, one slightly above and the other below, intersecting $\T$ at a point infinitely close to $\mathcal{C}_{\T}$. We denote them by  $\ell'$ and $\ell''$, respectively. It follows that $d_{\T_{\ell'}^-}(u,v)<\diam(\T)$ or $d_{\T_{\ell''}^-}(u,v)<\diam(\T)$.

\

\noindent \emph{Case 3.} \emph{All vertices of $P(v,\mathcal{C}_{\T})$ are in the same halfplane  determined by $\ell$, and line $\ell$ intersects the convex hull of $\T(v, \mathcal{C}_{\T})$}.
As explained before, each of the $O(n)$ classes of combinatorially equivalent lines that go through $\mathcal{C}_{\T}$ can be identified with a double wedge with apex $\mathcal{C}_{\T}$. Let $e$ be an edge of $\T(v, \mathcal{C}_{\T})$ that crosses one of the double wedges that correspond to $\ell$, say $\mathcal{W}$. We denote by  $\ell_{\mathcal{W}}$ and $\widetilde{\ell}_\mathcal{W}$ the extreme lines defining the wedge $\mathcal{W}$, where $\widetilde{\ell}_\mathcal{W}$ is the angularly closest to $v$ (from $\mathcal{C}_{\T}$); see Figure \ref{fig:wedge}a. This edge $e$ either:
\begin{enumerate}
    \item[(i)] shortens the distance (in $\T$) from $v$ to $ \mathcal{C}_{\T}$ using $\ell_{\mathcal{W}}$ (by the triangle inequality, it would also shorten that distance using any line in $\mathcal{W}$), or
    \item[(ii)] does not shorten the distance (in $\T$) from $v$ to $ \mathcal{C}_{\T}$ using $\widetilde{\ell}_\mathcal{W}$   (again, by triangle inequality, it would not shorten that distance  using any line in $\mathcal{W}$), or
    \item[(iii)] shortens the distance from $v$ to $ \mathcal{C}_{\T}$ (in $\T$) using $\widetilde{\ell}_\mathcal{W}$  but not using $\ell_{\mathcal{W}}$. Then, we can split $\mathcal{W}$ into two double wedges $\mathcal{W}_1$ and $\mathcal{W}\setminus \mathcal{W}_1 $ such that all lines of $\mathcal{W}$ shortening that distance using edge $e$ are in $\mathcal{W}_1$.

We claim that each edge $e$ contributes to the splitting of at most one double wedge, and the boundary between $\mathcal{W}_1$ and $\mathcal{W}\setminus \mathcal{W}_1 $ can be computed using the law of cosines. Indeed, consider the set $\mathcal{S}$ of lines that go through the center $ \mathcal{C}_{\T}$ and intersect edge $e$. We parametrize the edge $e$ by $\lambda$ so that each fixed value of $\lambda$ corresponds to a line in $\mathcal{S}$. 
Since $e$ is a fixed edge, the distance in $\T$ from one of its endpoints to $\mathcal{C}_{\T}$ is a fixed value, which we denote by $a$. 
Hence, for each value of $\lambda$, there is a unique line in $\mathcal{S}$ that can satisfy the law of cosines, as Figure~\ref{fig:wedge}b illustrates.
This means that one edge $e$ can cause the division of at most one double wedge  $\mathcal{W}$. 
Therefore, after all possible subdivisions, we still have $O(n)$ double wedges, and all lines in each double wedge have the same metric behavior. Moreover, for a given double wedge $\mathcal{W}$, the value that determines the boundary between $\mathcal{W}_1$ and $\mathcal{W}\setminus \mathcal{W}_1 $  (parametrizing the sub-segment with endpoints $\widetilde{\ell}_\mathcal{W}\cap e$ and ${\ell}_\mathcal{W}\cap e$) can be computed, by the law of cosines, as:
$$\sqrt{b^2+c^2-2bc \cos \gamma},$$
where $b$ (resp., $c$) is the length of the segment with endpoints $\mathcal{C}_{\T}$  and   $\widetilde{\ell}_\mathcal{W}\cap e$ (resp., ${\ell}_\mathcal{W}\cap e$), and $\gamma$ is the angle formed by $\ell_\mathcal{W}$ and $\widetilde{\ell}_\mathcal{W}$ in the wedge crossed by $e$.

\end{enumerate}

\begin{figure}
\centering
	\includegraphics{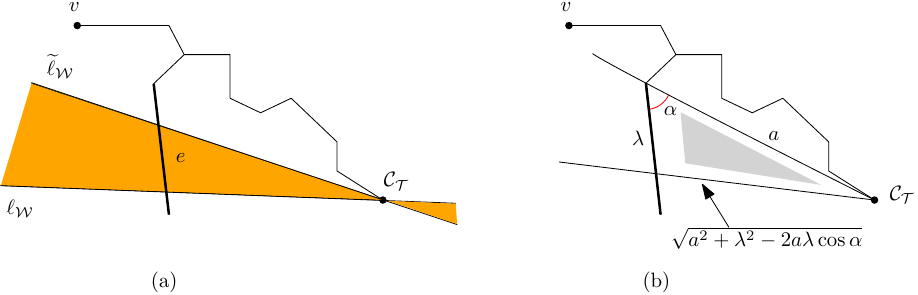}
     \caption{(a) Scheme of the setting of case 3 (the double wedge $\mathcal{W}$ in orange); (b) 
     Since the edge is fixed, the distance $a$ and the angle $\alpha$ are also fixed, and therefore, the law of cosines forces a unique value of $\lambda$.}
	\label{fig:wedge}
\end{figure}

 The conditions of cases 1 and 2 can be checked in linear time for each of the $O(n)$ combinatorially different classes of candidate lines (note that convex hulls do not need to be computed to check the conditions, since it is enough to check extreme lines). For each class, take any line $\ell$ as a representative. If we find a line in those cases, $b(\T)=2$. Otherwise, our candidate line for having Borsuk number $2$ must satisfy the conditions of case 3. Then, we define a matrix $\mathcal{M}(\T)$ whose columns are indexed by the diametral leaves $v$ of $\T$, and the rows by the double wedges $\mathcal{W}$; the entry $v\mathcal{W}$ is $0$ if the distance from $v$ to $\mathcal{C}_{\T}$ (in $\T$) is shortened using the lines of $\mathcal{W}$, and $\pm 1$ otherwise: $+1$ if $v$ is above  $\mathcal{W}$, and $-1$ if it is below $\mathcal{W}$ (i.e., one side and the other determined by the wedge). Since there are $O(n)$ leaves and $O(n)$ wedges, the matrix $\mathcal{M}(\T)$ has size $O(n^2)$. 
 
 The Borsuk number of $\T$ (in case 3) is $2$
if and only if there is a row in $\mathcal{M}(\T)$ with either at most one entry $+1$ or at most one entry $-1$. Indeed, a non-zero entry in $\mathcal{M}(\T)$ means that there is a diametral leaf $v$ and a class of combinatorially equivalent lines such that no line in the class decreases the distance from $v$ to $\mathcal{C}_{\T}$. Hence, a row of $\mathcal{M}(\T)$ with at most one $+1$ entry, say, at entry $v\mathcal{W}$, tells us that there is a class of combinatorially equivalent lines, associated to the double wedge $\mathcal{W}$, such that:
\begin{itemize}
\item no line in $\mathcal{W}$ decreases $d_{\T}(v, \mathcal{C}_{\T})$;
\item all lines in $\mathcal{W}$ decrease the distance (in $\T$) from $\mathcal{C}_{\T}$ to all diametral leaves with entry $0$;
\item all diametral leaves corresponding to entry $-1$ are on the other side (than $v$) of $\mathcal{W}$. 
\end{itemize}
By Lemma \ref{lem:throughcenter} we can take a line $\ell$ intersecting $\T$ at a point infinitely close to $\mathcal{C}_{\T}$, parallel to any line of the class, and leaving $\mathcal{C}_{\T}$ on the same half-plane as $v$. Thus, we divide $\T$ into the graphs $\T_{\ell}^+$ and $\T_{\ell}^-$ in such a way that $v$ and $\mathcal{C}_{\T}$ are in the same new graph, and all diametral leaves whose distance to the center is not decreased are in the other; suppose that $v, \mathcal{C}_{\T}\in \T_{\ell}^+$. It holds that ${\rm diam}(\T_{\ell}^+)<{\rm diam}(\T)$, since
the unique distance from a diametral leaf of $\T$ to $\mathcal{C}_{\T}$ that is not decreased in $\T_{\ell}^+$ is that from $v$, and to reach distance $\diam(\T)$, there must be at least two such diametral leaves. Moreover, the diameter of $\T_{\ell}^-$ is also strictly smaller than that of $\T$, since the center $\mathcal{C}_{\T}$ is located in  $\T_{\ell}^+$ and so, for any two diametral leaves of $\T$ that are in $\T_{\ell}^-$, their path in $\T$ is shortened by some segment contained in $\ell$.

Therefore, in each case, we can determine whether there is a line giving a correct partition in $O(n^2)$ time, if so, $b(\T)=2$; otherwise, $b(\T)=3$.
\end{proof}

\section{Concluding remarks}\label{sec:conclusions}

We extended the widely-studied Borsuk's problem to the context of graphs. Our approach either partitions an abstract graph by removing edges, or a continuous plane geometric graph by a sequence of line cuts. In both cases, one could propose other partitioning methods; however, we believe that ours are natural, since, as explained before, they preserve the core idea of Borsuk's original problem, and there are real-world applications (in, for example, social networks studies or network design) related with our approach. 

There are many questions that remain open.  A challenging problem is to improve our bounds for the Borsuk number of continuous geometric graphs; in particular, to determine whether this parameter can be upper bounded by a constant (in analogy with Borsuk's original problem). We believe that  further exploration of the structure of diametral sets in  plane continuous geometric graphs  can contribute to progress in this problem. 
Finally, it would also be interesting to design efficient algorithms to compute the Borsuk number, and to characterize the graphs with a given Borsuk number, for instance, with Borsuk number two.

 \paragraph{Acknowledgments.}
D.G., A.M., and R.S. were partially supported by grant PID2023-150725NB-I00 funded by MICIU/AEI/10.13039/501100011033.
D.G. was also supported by RED2024-153572-T.

\bibliographystyle{plainurl}
\bibliography{borsuk.bib}

\newpage
\section*{Appendix}

In this appendix we include the code used to prove Proposition~\ref{prop:disjdiam}.

The code follows the idea presented in the proof sketch. For technical reasons, the variables used in the linear program are named different here than in the proof sketch. 
Refer to Figure~\ref{fig:programCode} for an illustration of the main variables in the linear program (method \texttt{runLP}).
Firstly, the four points in questions ($p,q,p',q'$) are called $p,q,u,v$ in the program.
Secondly, in the linear program  we use single letters (`a' to `j') to name the variables  with the segments of paths along the two disjoint shortest paths, and `r',`s',`t',`u' for the lengths of the crossings parts.
Thirdly, entry and exist points along the $p$-$q$ and $u$-$v$ are numbered $0$--$3$.

\begin{figure}[tb]
    \centering
    \includegraphics{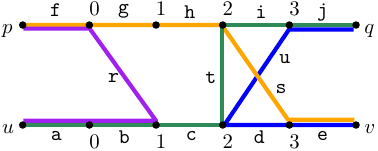}
    \caption{Illustration of variables in the linear program. The numbers $0$--$3$ are the indices of the possible exit/entry points on the two shortest paths.}
    \label{fig:programCode}
\end{figure}

The complete check for the $4^8$ cases is carried out when running last piece of code, marked as ``MAIN CODE''.

\begin{lstlisting}[caption={Full SageMath Implementation}]
from sage.graphs.planarity import is_planar
#### Solves one instance of the two diameter problem
#### Arguments:
#### inst_no: instance number (used only for reporting)
#### pu_in, pu_out, pv_in, pv_out, qu_in, qu_out, qv_in, qv_out: entry and exit points (0,1,2,or 3) for each of the four crossing paths, always from the point of view of diameter pq (pu_in=1, pu_out=3 means that the path from p to u goes from p to p_1, then it crosses to p_3 on uv, and then goes from p_3 to u).
#### show_values: boolean value to show variable values in case of feasible solution
def runLP(inst_no, pu_in, pu_out, pv_in, pv_out, qu_in, qu_out, qv_in, qv_out,show_values):
    p = MixedIntegerLinearProgram()
    v = p.new_variable(real=True,nonnegative=True)
    a,b,c,d,e,f,g,h,i,j = v['a'], v['b'], v['c'],v['d'], v['e'],v['f'], v['g'], v['h'], v['i'], v['j']
    r,s,t,u = v['r'], v['s'],v['t'], v['u'] #lenghts of the four crossing connections
    m = v['m']
    D = v['D']

    # Partial cost tableS
    pq_costs_from_p = [f,f+g,f+g+h,f+g+h+i]
    pq_costs_from_q = [g+h+i+j,h+i+j,i+j,j]
    uv_costs_from_u = [a,a+b,a+b+c,a+b+c+d]
    uv_costs_from_v = [b+c+d+e,c+d+e,d+e,e]

    # The two diameters
    p.add_constraint(D == 10) # WLOG we assume this is the value of the diameter
    p.add_constraint(a+b+c+d+e == D)
    p.add_constraint(f+g+h+i+j == D)

    # Four crossing shortest paths
    p.add_constraint(pq_costs_from_p[pu_in]+ r + uv_costs_from_u[pu_out] <= D) #p~u
    p.add_constraint(pq_costs_from_p[pv_in]+ s + uv_costs_from_v[pv_out] <= D) #p~v
    p.add_constraint(pq_costs_from_q[qu_in]+ t + uv_costs_from_u[qu_out] <= D) #q~u
    p.add_constraint(pq_costs_from_q[qv_in]+ u + uv_costs_from_v[qv_out] <= D) #q~v

    # Connections between the two diameters must be >0
    p.add_constraint(r>=m)
    p.add_constraint(s>=m)
    p.add_constraint(t>=m)
    p.add_constraint(u>=m)
    # Also the first and last edge of each diameter ("hairs") have to be >0
    p.add_constraint(a>=0) #antes >=m
    p.add_constraint(e>=0)
    p.add_constraint(f>=0)
    p.add_constraint(j>=0)

    #Add all path constraints
    # Two arrays, one for the entry points along uv, and one for those along pq
    uv_entries = {0: [], 1: [], 2: [], 3: []}
    pq_entries = {0: [], 1: [], 2: [], 3: []}

    # Fill in all the information about connections
    pq_entries[pu_in].append({'out':pu_out, 'cost': r})
    uv_entries[pu_out].append({'out':pu_in, 'cost': r})
    pq_entries[pv_in].append({'out':pv_out, 'cost': s})
    uv_entries[pv_out].append({'out':pv_in, 'cost': s})
    pq_entries[qu_in].append({'out':qu_out, 'cost': t})
    uv_entries[qu_out].append({'out':qu_in, 'cost': t})
    pq_entries[qv_in].append({'out':qv_out, 'cost': u})
    uv_entries[qv_out].append({'out':qv_in, 'cost': u})

    # pq-paths: Add a constraint >D for each path that start at p, goes to uv, goes back to pq, and ends at q
    for entry in pq_entries:
        for exit in pq_entries[entry]:
            for entryuv in uv_entries:
                for exituv in uv_entries[entryuv]:
                    pq_in1=entry
                    uv_in1=exit['out']
                    uv_in2=entryuv
                    pq_in2=exituv['out']
                    # Path on uv between uv_in1 and uv_in2
                    if (uv_in1<uv_in2):
                        cost_uv = uv_costs_from_u[uv_in2] - uv_costs_from_u[uv_in1]
                    else:
                        cost_uv = uv_costs_from_u[uv_in1] - uv_costs_from_u[uv_in2]
                    total_cost = pq_costs_from_p[pq_in1] + exit['cost'] + cost_uv + exituv['cost'] + pq_costs_from_q[pq_in2]

                    p.add_constraint(total_cost >= D+m)

    # uv-paths: Add a constraint >D for each path that start at u, goes to pq, goes back to uv, and ends at v
    for entry in uv_entries:
        for exit in uv_entries[entry]:
            for entrypq in pq_entries:
                for exitpq in pq_entries[entrypq]:
                    uv_in1=entry
                    pq_in1=exit['out']
                    pq_in2=entrypq
                    uv_in2=exitpq['out']
                    # Path on pq between pq_in1 and pq_in2
                    if (pq_in1<pq_in2):
                        cost_pq = pq_costs_from_p[pq_in2] - pq_costs_from_p[pq_in1]
                    else:
                        cost_pq = pq_costs_from_p[pq_in1] - pq_costs_from_p[pq_in2]
                    total_cost = uv_costs_from_u[uv_in1] + exit['cost'] + cost_pq + exitpq['cost'] + uv_costs_from_v[uv_in2]

                    p.add_constraint(total_cost >= D+m)

    p.set_objective(m) # There will be a feasible solution iff the objective function is >0
    p.solve()

    #Report result if m>0
    if (p.get_values(m)>0 or show_values):
        # Build extended graph to check planarity
        G = buildExtendedGraph(pu_in, pu_out, pv_in, pv_out, qu_in, qu_out, qv_in, qv_out)
        # G is planar iff p,q,u,v are on the same face
        if (sage.graphs.planarity.is_planar(G)):
            print ("Instance #"+str(inst_no) +" "+ str([pu_in, pu_out, pv_in, pv_out, qu_in, qu_out, qv_in, qv_out])+ " -> m="+ str(p.get_values(m)))
            print("Planar")

    return (p.get_values(m)>0 )

# Builds the graph from the diagonal choices, adding one extra node connected to p,q,u,v. This graph will be planar if and only if p,q,u,v lie on the same face.

def buildExtendedGraph(pu_in, pu_out, pv_in, pv_out, qu_in, qu_out, qv_in, qv_out):
    # The two diameter paths, p=1, q=2, u=3, v=4
    # Path pq: p=1,10,11,12,13,2=q
    # Path uv: u=3,20,21,22,23,4=v
    G = Graph({1:[10], 10:[11], 11:[12], 12:[13], 13:[2], 3:[20], 20:[21], 21:[22], 22:[23], 23:[4]})
    # The four crossing connections
    G.add_edge((10+pu_in, 20+pu_out))
    G.add_edge((10+pv_in, 20+pv_out))
    G.add_edge((10+qu_in, 20+qu_out))
    G.add_edge((10+qv_in, 20+qv_out))
    # One extra vertex connected to u,v,p,q
    G.add_edge(1, 99)
    G.add_edge(2, 99)
    G.add_edge(3, 99)
    G.add_edge(4, 99)
    return G

## MAIN CODE: try out ALL 4^8 combinations

#pu_in, pu_out, pv_in, pv_out, qu_in, qu_out, qv_in, qv_out
c=0
found=0
print("Starting all case analysis.")
for pu_in in [0,1,2,3]:
    for pu_out in [0,1,2,3]:
        for pv_in in [0,1,2,3]:
            for pv_out in [0,1,2,3]:
                for qu_in in [0,1,2,3]:
                    for qu_out in [0,1,2,3]:
                        for qv_in in [0,1,2,3]:
                            for qv_out in [0,1,2,3]:
                                c=c+1 #instance counter
                                feasible=runLP(c, pu_in, pu_out, pv_in, pv_out, qu_in, qu_out, qv_in, qv_out, False)
                                if (feasible):
                                    found=found+1
                                    print ("#found: "+str(found))
    print("Mega-iteration puin="+str(pu_in)+" finished")
print("Finished.")
\end{lstlisting}

\end{document}